\documentclass[pdflatex,sn-mathphys-num]{sn-jnl}
\usepackage{enumitem}
\usepackage{graphicx}%
\usepackage{multirow}%
\usepackage[all]{xy}
\usepackage{amsmath,amssymb,amsfonts}%
\usepackage{amsthm}%
\usepackage{mathrsfs}%
\usepackage[title]{appendix}%
\usepackage{xcolor}%
\usepackage{textcomp}%
\usepackage{manyfoot}%
\usepackage{booktabs}%
\usepackage{algorithm}%
\usepackage{algorithmicx}%
\usepackage{algpseudocode}%
\usepackage{listings}%
\usepackage{soul}
\usepackage{enumitem}
\usepackage{tikz-cd}
\theoremstyle{thmstyleone}%
\newtheorem{theorem}{Theorem}[section]%  meant for continuous numbers
\numberwithin{equation}{section}
\newtheorem{proposition}[theorem]{Proposition}% 

\theoremstyle{thmstyletwo}%
\newtheorem{remark}[theorem]{Remark}%

\theoremstyle{thmstylethree}%
\newtheorem{definition}[theorem]{Definition}%

\newtheorem{corollary}[theorem]{Corollary}
\begin{document}

    \title[Intermediate topological pressures]{Intermediate topological pressures and variational principles for nonautonomous dynamical systems}

%%=============================================================%%
%% GivenName	-> \fnm{Joergen W.}
%% Particle	-> \spfx{van der} -> surname prefix
%% FamilyName	-> \sur{Ploeg}
%% Suffix	-> \sfx{IV}
%% \author*[1,2]{\fnm{Joergen W.} \spfx{van der} \sur{Ploeg} 
%%  \sfx{IV}}\email{iauthor@gmail.com}
%%=============================================================%%

\author[1]{\fnm{Yujun} \sur{Ju}} \email{yjju@ctbu.edu.cn}

\affil[1]{\orgdiv{School of Mathematics and Statistics}, \orgname{Chongqing Technology and Business University}, \orgaddress{\city{Chongqing}, \postcode{400067}, \country{People's Republic of China}}}

%%==================================%%
%% Sample for unstructured abstract %%
%%==================================%%

\abstract{
We introduce a one-parameter family of intermediate topological pressures for nonautonomous dynamical systems that interpolates between the Pesin-Pitskel topological pressure and the lower and upper capacity topological pressures.
The construction is based on the Carath\'eodory-Pesin structure in which all admissible strings in a covering satisfy $N \le n < N/\theta + 1$, where $\theta \in [0,1]$ is a parameter. 
The extremal cases $\theta=0$ and $\theta=1$ recover the Pesin-Pitskel pressure and the two capacity pressures, respectively.
We first investigate several properties of the intermediate pressure, including continuity on $(0,1]$ and possible discontinuity at $\theta=0$, as well as the power rule and monotonicity.
We then derive inequalities for intermediate pressures with respect to a factor map. Finally, we introduce intermediate measure-theoretic pressures and establish variational principles relating them to the corresponding topological pressures.}



\maketitle

\section{Introduction}
Topological entropy is one of the most fundamental invariants in dynamical systems, measuring the exponential growth rate of distinguishable orbits.
It was first introduced by Adler, Konheim and McAndrew \cite{adler1965topological} through open covers.
Later, Bowen \cite{bowen1971entropy} and Dinaburg \cite{dinaburg1970correlation} provided equivalent definitions based on spanning and separated sets, in a way parallel to the definition of the (upper) box dimension.
Bowen subsequently introduced the notion of topological entropy on subsets \cite{bowen1973topological}, using a construction that resembles the definition of the Hausdorff dimension.
Pesin \cite{pesin1997dimension} developed a refinement of the classical Carath\'eodory construction, now generally known as the Carath\'eodory-Pesin structure.
This framework has become a central tool in the study of dynamical systems and dimension theory.
It gives a unified way to describe the Hausdorff and box dimensions, topological entropy and topological pressure for non-compact sets.
Feng and Huang \cite{feng2012variational} introduced the packing topological entropy as the dynamical analogue of the packing dimension, and proved variational principles for both Bowen and packing topological entropies.
These developments illustrate the close connections between fractal dimensions and topological entropies.

The classical topological entropy for nonautonomous dynamical systems (NDSs for short) was first introduced and studied by Kolyada and Snoha \cite{kolyada1996topological}.
Li \cite{li2015remarks} subsequently used the Carath\'eodory-Pesin structure to define the Pesin topological entropy of NDSs on non-compact sets and provided a condition under which it coincides with the classical one on the whole space. Li and Ye \cite{li2023comparison} later
obtained another criterion for this equality by showing that the Pesin and classical entropies agree whenever the system is weakly mixing.
Bi\'s \cite{bis2018topological} employed the Carath\'eodory-Pesin structure to define the upper capacity topological entropy for NDSs on non-compact sets and proved that it agrees with the classical entropy on every subset.  
Variational principles for Bowen and packing topological entropies of NDSs were obtained by Xu and Zhou \cite{xu2018variational} and by Zhang and Zhu \cite{zhang2023variational}. 
Along a different line of development, several finer invariants have been 
introduced to distinguish NDSs with zero topological entropy, 
including topological entropy dimension \cite{kuang2013fractal, li2019topological, li2025dimension}, 
polynomial entropy \cite{liu2022polynomial} and topological sequence entropy 
\cite{shao2026topological}.
In addition, mean dimension and metric mean dimension were extended to NDSs in 
\cite{rodrigues2022mean}, providing tools for classifying systems with infinite 
topological entropy.

Falconer, Fraser and Kempton \cite{falconer2020intermediate} introduced the intermediate dimensions, a one-parameter family of dimensions depending on $\theta \in [0,1]$.
The construction is based on coverings whose diameters are restricted to the interval $[\delta^{1/\theta}, \delta]$, so that the Hausdorff and box dimensions appear as the extreme cases when $\theta = 0$ and $\theta = 1$. Intermediate dimensions enjoy several useful properties. 
They are continuous on $(0,1]$ (though not necessarily at $0$), 
satisfy analogues of the mass distribution principle, Frostman's lemma and product formulas, and provide insight into the distribution of covering scales for sets whose Hausdorff and box dimensions differ, 
offering a refined description of geometric complexity. Motivated by this work, we recently introduced the lower and upper $\theta$-intermediate topological entropies for NDSs and studied their dependence on the parameter $\theta$ \cite{ju2026entropy}.  
The intermediate entropies mirror the continuity behaviour of the intermediate 
dimensions: they are continuous on $(0,1]$ but may fail to be continuous at $0$.  
An explicit example demonstrating discontinuity at $0$ was also obtained. Liu, Selmi and Li \cite{liu2026intermediate} considered a related notion of $\theta$-intermediate topological entropy for autonomous dynamical systems, where admissible string lengths satisfy $N \le n \le N^{1/\theta}$, leading to a different interpolation scheme.

Topological pressure was introduced by Ruelle \cite{ruelle1973statistical} 
and later studied for continuous maps on compact spaces by Walters 
\cite{walters1982ergodic}.  
It extends topological entropy and plays a central role in the thermodynamic 
formalism, providing a finer description of dynamical complexity by 
incorporating potential functions.
Pesin and Pitskel \cite{pesin1984topological} extended Bowen’s subset entropy to a corresponding notion of topological pressure on subsets. This notion is also referred to as Pesin-Pitskel topological pressure (or Bowen topological pressure).
Tang, Cheng and Zhao \cite{tang2015variational} and Zhong and Chen \cite{zhong2023variational} extended the work of Feng and Huang and established variational principles for the Pesin-Pitskel topological pressure and the packing topological pressure, respectively.
For NDSs, variational principles for Pesin-Pitskel and packing topological pressures were obtained by Nazarian Sarkooh \cite{nazarian2024variational} and Li \cite{li2024variational}, respectively.
More recently, Chen and Miao \cite{chen2025nonautonomous, chen2025nonautonomousII} carried out a detailed and systematic study of various topological entropies and pressures for more general nonautonomous systems in which both the state spaces and the potentials vary with time, and established the corresponding variational principles in this more general setting.

These results naturally lead to the question of whether one can construct a 
family of topological pressures, depending on a parameter $\theta$, that 
interpolates between the Pesin-Pitskel pressure and the capacity pressures 
in a way analogous to the intermediate dimensions and intermediate entropies.  
Motivated by this question, we introduce the lower and upper 
$\theta$-intermediate topological pressures for NDSs, which interpolate between the Pesin-Pitskel topological pressure and the lower and upper capacity topological pressures. We establish fundamental properties of these pressures, including continuity with respect to $\theta$ on $(0,1]$, the power rule, monotonicity and their behaviour under factor maps. 
We also introduce the corresponding $\theta$-intermediate measure-theoretic 
pressures and prove variational principles that relate them to the 
associated topological pressures.

The paper is organized as follows. 
In Section~\ref{sec:definition}, we introduce the lower and upper $\theta$-intermediate topological pressures for NDSs, provide two equivalent definitions and establish quantitative inequalities that extend the continuity estimates previously obtained for intermediate topological entropies. 
In Section~\ref{sec:properties}, we discuss several fundamental properties of the 
$\theta$-intermediate topological pressures, including closure stability,  the power rule and monotonicity.
In Section~\ref{sec:factor}, we study the relations between the $\theta$-intermediate 
topological pressures of two topologically semiconjugate systems and 
obtain inequality formulas for $\theta$-intermediate pressures via a factor map.
In Section~\ref{sec:variational}, we introduce intermediate measure-theoretic pressures and prove variational principles relating them to the corresponding topological pressures.

\section{Intermediate topological pressures: definition and basic properties}\label{sec:definition}

Let \( (X,d) \) be a compact metric space and let \( \boldsymbol{f}= \{f_i\}_{i=1}^\infty \) be a sequence of continuous self-maps of \( X \). 
Denote by \(\mathbb{N}\) the set of positive integers and by \(\mathbb{R}\) the set of real numbers.
For each \( i \in \mathbb{N} \), set \( f_i^0 = \text{id}_X \), the identity map on \( X \), and for each \( n \in \mathbb{N} \) define
\[
f_i^n = f_{i+(n - 1)} \circ \cdots \circ f_{i+1} \circ f_i, \quad f_i^{-n} = \left(f_i^n\right)^{-1} = f_i^{-1} \circ f_{i+1}^{-1} \circ \cdots \circ f_{i+(n - 1)}^{-1}.
\]
The notation \( f_i^{-n} \) is used for preimages of sets, although we do not assume that the maps \( f_i \) are invertible.
Then we call \( (X, \boldsymbol{f}) \) a nonautonomous dynamical system (NDS for short).
Finally, denote by $\boldsymbol{f}^n$ the sequence of maps $\{f_{in+1}^{n}\}_{i=0}^\infty$ and by $\boldsymbol{f}_n$ the sequence $\{f_{i}\}_{i=n}^\infty$. 
Set
\[
d_n(x,y)=\max_{0\le j\le n-1}d\bigl(f_1^j(x),f_1^j(y)\bigr),\qquad x,y\in X.
\]
Since \(X\) is compact, \(d_n\) is a metric equivalent to \(d\).  
Given $\varepsilon>0$ and $x\in X$, the $(n,\varepsilon)$-Bowen ball is
\[
B_n(x,\varepsilon)= \left\{y\in X : d_n(x,y)<\varepsilon \right\}.
\]

In what follows, we recall the definition of the topological pressure of an NDS on a nonempty subset, using spanning sets and separated sets \cite{huang2008pressure,kong2015topological, li2022notes}.

Let $Z$ be a nonempty subset of $X$. The set $Z$ may not be compact and may not exhibit any kind of invariancy with respect to $\boldsymbol{f}$. 
A set \(E\subseteq X\) is an \((n,\varepsilon)\)-spanning set of \(Z\) if for every \(y\in Z\) there exists \(x\in E\) with \(d_n(x,y)\le\varepsilon\);
a set \(F\subseteq Z\) is an \((n,\varepsilon)\)-separated set of \(Z\) if \(x\neq y\) in \(F\) implies \(d_n(x,y)>\varepsilon\).
Let $C(X, \mathbb{R})$ denote the Banach space of all continuous real-valued functions on $X$ equipped with the supremum norm.
Given $\varphi \in C(X, \mathbb{R})$, we define
\[
S_n^{\boldsymbol{f}}\varphi(x):=\sum_{j=0}^{n-1}\varphi\bigl(f_1^j x\bigr), \qquad x\in X.
\]
For simplicity, we write \(S_n\varphi(x)\) instead of \(S_n^{\boldsymbol f}\varphi(x)\) whenever no confusion arises.
For \(t>0\), we also denote the modulus of continuity of \(\varphi\) by
\[
\omega_\varphi(t)
:=
\sup\{|\varphi(x)-\varphi(y)|: d(x,y)<t\}.
\]
Since \(X\) is compact and \(\varphi\) is continuous, we have \(\omega_\varphi(t)\to0\) as $t\to0$.
When the potential is fixed, we write \(\omega(t)\) for \(\omega_\varphi(t)\).

For any $\varepsilon > 0$, define
\[
Q_n(\boldsymbol{f}, Z, \varphi, \varepsilon) = \inf \left\{ \sum_{x \in E} e^{S_n \varphi(x)} : E \text{ is an } (n, \varepsilon)\text{-spanning set for } Z \right\},
\]
\[
P_n(\boldsymbol{f}, Z, \varphi, \varepsilon) = \sup \left\{ \sum_{x \in F} e^{S_n \varphi(x)} : F \text{ is an } (n, \varepsilon)\text{-separated set of } Z \right\}.
\]
Then set
\[
Q(\boldsymbol{f}, Z, \varphi, \varepsilon) = \limsup_{n \to \infty} \frac{1}{n} \log Q_n(\boldsymbol{f}, Z, \varphi, \varepsilon),
\]
\[
P(\boldsymbol{f}, Z, \varphi, \varepsilon) = \limsup_{n \to \infty} \frac{1}{n} \log P_n(\boldsymbol{f},Z, \varphi, \varepsilon).
\]
\begin{definition}
Let \(Z\subseteq X\) be nonempty. The \textit{classical topological pressure} of the function \( \varphi \) on the set \( Z \) with respect to \( \boldsymbol{f} \) is given by
\[
P(\boldsymbol{f}, Z, \varphi) = \lim_{\varepsilon \to 0} Q(\boldsymbol{f}, Z, \varphi, \varepsilon) = \lim_{\varepsilon \to 0} P(\boldsymbol{f}, Z, \varphi, \varepsilon). 
\]
In particular, when \( \varphi \equiv 0 \), this reduces to the classical topological entropy for NDSs:
\(
h(\boldsymbol{f}, Z) = P(\boldsymbol{f},Z, 0),
\)
as introduced by Kolyada and Snoha \cite{kolyada1996topological}.
\end{definition}

\subsection{Intermediate topological pressures of NDSs}
Let $(X,\boldsymbol{f})$ be an NDS on a compact metric space $(X,d)$.
For a finite open cover $\mathcal{U}$ of $X$ and \(m \in \mathbb{N}\), let
\[
\mathcal{S}_{m}(\mathcal{U}):=\left\{\mathbf{U}=(U_{0},U_{1},\ldots,U_{m-1}):\mathbf{U}\in \mathcal{U}^{m}\right\},
\]
where $\mathcal{U}^{m}=\prod\limits_{i=1}^{m}\mathcal{U}$. For any string $\mathbf{U}\in \mathcal{S}_{m}(\mathcal{U})$, define the length of $\mathbf{U}$ to be $m(\mathbf{U}):=m$. We put \(\mathcal{S}=\mathcal{S}(\mathcal{U})=\bigcup_{m \in\mathbb N}\mathcal{S}_{m}(\mathcal{U})\).
If $k \in \mathbb{N}$ with $1 \le k \le m(\mathbf{U})$ and $0 \le a \le m(\mathbf{U}) - k$, we denote by
\[
\mathbf U|_{[a,a+k-1]} := (U_a,U_{a+1},\ldots,U_{a+k-1}) \in \mathcal S_k(\mathcal U)
\]
the substring of $\mathbf U$ of length $k$ starting at position $a$. 
In particular, $\mathbf U|_{[0,k-1]}$ is the initial truncation of length $k$. 

For a given string $\mathbf{U}=(U_0,U_1,\ldots,U_{m-1}) \in \mathcal{S}_{m}(\mathcal{U})$, we associate the set
\[
 X_{\boldsymbol{f}}(\mathbf{U})=\left\{x\in X : f_1^{j}(x)\in U_j, j=0,1,\ldots,m-1\right\}.
\]
When no confusion arises we simply write \(X(\mathbf U)\) for \(X_{\boldsymbol{f}}(\mathbf U)\),
and likewise omit the subscript \(\boldsymbol{f}\) from the quantities \(M\), \(\underline m\), \(\overline m\) defined below.
Let \(\varphi \in C(X,\mathbb{R})\).
For any subset $Z\subseteq X, \alpha \in \mathbb{R}$ and $\theta\in[0,1]$, define
\[
\underline{m}(Z,\alpha,\varphi,\mathcal{U},\theta)=\liminf_{N\to\infty}M(Z,\alpha,\varphi,\mathcal{U},N,\theta),
\]
\[
\overline{m}(Z,\alpha,\varphi,\mathcal{U},\theta)=\limsup_{N\to\infty}M(Z,\alpha,\varphi,\mathcal{U},N,\theta),
\]
where
\[
M(Z,\alpha,\varphi,\mathcal{U},N,\theta):=\inf_{\mathcal{G}}\left\{\sum\limits_{\mathbf{U} \in \mathcal{G}}\exp\biggl(-\alpha m(\mathbf{U})+\sup_{x \in X(\mathbf U)}S_{m(\mathbf{U})}\varphi(x)\biggr)\right\}
\]
and the infimum is taken over all finite or countable collections of strings $\mathcal{G} \subseteq \bigcup_{N \leq m < N/\theta+1} \mathcal{S}_{m}(\mathcal{U})$ such that $ \mathcal{G}$ covers $Z$ $\left({\rm i.e.,} \bigcup_{\mathbf{U}\in \mathcal{G}}X(\mathbf U)\supseteq Z\right)$. Here and throughout the paper, we adopt the convention that
\[
\sup_{x\in X(\mathbf U)} S_{m(\mathbf U)}\varphi(x)=-\infty
\quad \text{whenever } X(\mathbf U)=\emptyset.
\]

It is straightforward to verify that the critical values of 
$\underline{m}(Z,\alpha,\varphi,\mathcal{U},\theta)$ and 
$\overline{m}(Z,\alpha,\varphi,\mathcal{U},\theta)$ exist. We define
\[
\underline{P}(\boldsymbol{f},Z,\varphi,\mathcal{U},\theta)
:=\inf\{\alpha: \underline{m}(Z,\alpha,\varphi,\mathcal{U},\theta)=0\}
=\sup\{\alpha: \underline{m}(Z,\alpha,\varphi,\mathcal{U},\theta)=\infty\},
\]
\[
\overline{P}(\boldsymbol{f},Z,\varphi,\mathcal{U},\theta)
:=\inf\{\alpha: \overline{m}(Z,\alpha,\varphi,\mathcal{U},\theta)=0\}
=\sup\{\alpha: \overline{m}(Z,\alpha,\varphi,\mathcal{U},\theta)=\infty\}.
\]

For an open cover $\mathcal{U}$ of $X$, let $|\mathcal U|= \max\{\operatorname{diam}(U) : U \in \mathcal{U}\}$. 

\begin{theorem}
For any nonempty subset $Z\subseteq X$, the following limits exist.
\[
\underline{P}(\boldsymbol{f},Z,\varphi,\theta)
=\lim_{|\mathcal U|\to 0}\,\underline{P}(\boldsymbol{f},Z,\varphi,\mathcal U,\theta),\]
\[
\overline{P}(\boldsymbol{f},Z,\varphi,\theta)
=\lim_{|\mathcal U|\to 0}\,\overline{P}(\boldsymbol{f},Z,\varphi,\mathcal U,\theta).
\]
\end{theorem}

\begin{proof}
We follow the classical argument of Pesin~\cite{pesin1997dimension}.
Let $\mathcal U$ and $\mathcal V$ be finite open covers of $X$ such that 
$|\mathcal V|$ is smaller than the Lebesgue number of $\mathcal U$. 
Then for every $V \in \mathcal V$ there exists $U(V)\in\mathcal U$ with $V\subseteq U(V)$. For any string 
\[
\mathbf V=(V_0,\ldots,V_{m-1})\in\mathcal S_m(\mathcal V),
\]
define the associated string
\[
\mathbf U(\mathbf V):=(U(V_0),\ldots,U(V_{m-1}))\in\mathcal S_m(\mathcal U).
\]
If $\mathcal G\subseteq \mathcal S(\mathcal V)$ covers $Z$, then
\(
\mathbf U(\mathcal G):=\{\mathbf U(\mathbf V):\mathbf V\in\mathcal G\}
\subseteq \mathcal S(\mathcal U)
\)
also covers $Z$. Without loss of generality, we may assume that $X(\mathbf V)\neq \emptyset$ for all $\mathbf V\in\mathcal G$. For each $\mathbf V\in\mathcal G$, fix $y_{\mathbf V}\in X(\mathbf V)$. Let
\[
\gamma=\gamma(\mathcal U):=
\sup\{|\varphi(x)-\varphi(y)|:x,y\in U \text{ for some } U\in\mathcal U\}.
\]
Then for any $x\in X(\mathbf U(\mathbf V))$, we have
\[
|S_m\varphi(x)-S_m\varphi(y_{\mathbf V})|\le m\gamma,
\]
and hence
\[
\sup_{x\in X(\mathbf U(\mathbf V))} S_m\varphi(x)
\le \sup_{y\in X(\mathbf V)} S_m\varphi(y)+m\gamma.
\]
Using the definition of $M(Z,\alpha,\varphi,\mathcal U,N,\theta)$, we obtain
\[
M(Z,\alpha,\varphi,\mathcal U,N,\theta)
\le M(Z,\alpha-\gamma,\varphi,\mathcal V,N,\theta)
\]
for all $\alpha\in\mathbb R$ and $N\in\mathbb N$. Consequently,
\[
\underline P(\boldsymbol f,Z,\varphi,\mathcal U,\theta)-\gamma
\le \underline P(\boldsymbol f,Z,\varphi,\mathcal V,\theta).
\]
Since $X$ is compact, it admits open covers of arbitrarily small diameter. Thus,
\[
\underline P(\boldsymbol f,Z,\varphi,\mathcal U,\theta)-\gamma
\le \liminf_{|\mathcal V|\to0}
\underline P(\boldsymbol f,Z,\varphi,\mathcal V,\theta).
\]
Letting $|\mathcal U|\to0$ yields $\gamma(\mathcal U)\to0$, and hence
\[
\limsup_{|\mathcal U|\to0}
\underline P(\boldsymbol f,Z,\varphi,\mathcal U,\theta)
\le
\liminf_{|\mathcal V|\to0}
\underline P(\boldsymbol f,Z,\varphi,\mathcal V,\theta).
\]
This proves the existence of the first limit. The second limit can be obtained by an analogous argument.
\end{proof}

We call the quantities 
\(\underline{P}(\boldsymbol{f},Z,\varphi,\theta)\) and 
\(\overline{P}(\boldsymbol{f},Z,\varphi,\theta)\) 
the \emph{lower} and \emph{upper \(\theta\)-intermediate topological pressures} 
of the function \(\varphi\) on the set \(Z\) with respect to \(\boldsymbol{f}\). 
If these two values coincide, we refer to the common value as the 
\emph{\(\theta\)-intermediate topological pressure} and denote it by 
\(P(\boldsymbol{f},Z,\varphi,\theta)\). In particular, when $\varphi \equiv 0$, the lower and upper $\theta$-intermediate topological pressures reduce to the lower and upper $\theta$-intermediate topological entropies on $Z$, which we denote by
\(
\underline{h}_{\mathrm{top}}(\boldsymbol{f},Z,\theta)\) and
\(\overline{h}_{\mathrm{top}}(\boldsymbol{f},Z,\theta)
\)
respectively. If these two values coincide, we write
\(
h_{\mathrm{top}}(\boldsymbol{f},Z,\theta)
\)
for the common $\theta$-intermediate topological entropy on $Z$.

\begin{remark}
(i)  When \(\theta=0\), the admissible strings satisfy \(m(\mathbf U)\ge N\). 
In this case, \(M(Z,\alpha,\varphi,\mathcal U,N,0)\) is non-decreasing with respect to \(N\), and hence
\[
\underline{P}(\boldsymbol{f},Z,\varphi,0)
=\overline{P}(\boldsymbol{f},Z,\varphi,0).
\]
Following \cite{nazarian2024variational}, we call this the \emph{Pesin-Pitskel topological pressure} and denote the common value by \(P^{B}(\boldsymbol{f},Z,\varphi)\). Furthermore, when the potential function \(\varphi \equiv 0\), 
the Pesin-Pitskel topological pressure \(P^{B}(\boldsymbol{f},Z,\varphi)\) 
reduces to 
$h_{\mathrm{top}}^{B}(\boldsymbol{f},Z)=P^{B}(\boldsymbol{f},Z,0),$
which is referred to as the \emph{Pesin topological entropy}, first introduced by Li \cite{li2015remarks}.

(ii) When $\theta=1$, the admissible strings satisfy \(m(\mathbf U)=N\). We recall that for a finite open cover $\mathcal{U}$ of $X$, the following identities hold:
\begin{equation}\label{eq:pressure-liminf-limsup}
\begin{aligned}
\underline{P}(\boldsymbol{f},Z,\varphi,\mathcal{U},1)
&= \liminf_{N\to\infty} \frac{1}{N} 
\log \Lambda(Z,\varphi,\mathcal{U},N),\\[4pt]
\overline{P}(\boldsymbol{f},Z,\varphi,\mathcal{U},1)
&= \limsup_{N\to\infty} \frac{1}{N} 
\log \Lambda(Z,\varphi,\mathcal{U},N),
\end{aligned}
\end{equation}
where
\[
\Lambda(Z,\varphi,\mathcal{U},N)
:= \inf_{\mathcal{G}} 
\left\{
\sum_{\mathbf{U}\in \mathcal{G}} 
\exp\!\left(\sup_{x\in X(\mathbf{U})} S_N \varphi(x)\right)
\right\},
\]
and the infimum is taken over all finite or countable 
$\mathcal{G} \subseteq \mathcal{S}_N(\mathcal{U})$ such that \(\mathcal{G}\) covers \(Z\) (cf. Yang and Huang~\cite[Lemma~3.5]{yang2025topological}).
Consequently, the lower and upper \emph{capacity topological pressures} of $\varphi$ on $Z$ (with respect to $\boldsymbol{f}$) are defined as
\[
\underline{CP}(\boldsymbol{f},Z,\varphi)
:= \underline{P}(\boldsymbol{f},Z,\varphi,1),
\qquad
\overline{CP}(\boldsymbol{f},Z,\varphi)
:= \overline{P}(\boldsymbol{f},Z,\varphi,1).
\]
Moreover, it was shown in \cite[Theorem~3.4]{li2024variational} that
\[
\overline{CP}(\boldsymbol{f},Z,\varphi)
=P(\boldsymbol{f},Z,\varphi),
\]
where $P(\boldsymbol{f},Z,\varphi)$ is the classical topological pressure.
\end{remark}

For \(\theta\in(0,1]\), if we restrict the admissible lengths of strings in the definition of
\(M(Z,\alpha,\varphi,\mathcal U,N,\theta)\) to the interval \([N,\,N/\theta]\),
then we obtain new functions \(M^{*}, \underline{m}^{*}\) and \(\overline{m}^{*}\),
and denote the corresponding critical values respectively by
\[
\underline{P}^{*}(\boldsymbol f,Z,\varphi,\mathcal U,\theta)
\quad\text{and}\quad
\overline{P}^{*}(\boldsymbol f,Z,\varphi,\mathcal U,\theta),
\]
which simplifies later computations by avoiding additional adjustments.

\begin{theorem}\label{thm:window-equivalence-nds}
Let $(X,d)$ be a compact metric space, $\boldsymbol f=\{f_n\}_{n=1}^\infty$ 
a sequence of continuous self-maps of $X$, $\mathcal U$ a finite open cover of $X$, 
and $\varphi\in C(X,\mathbb R)$. 
Then for any nonempty $Z\subseteq X$ and $\theta\in(0,1]$, one has
\[
\underline{P}^{*}(\boldsymbol f,Z,\varphi,\mathcal U,\theta)
=
\underline{P}(\boldsymbol f,Z,\varphi,\mathcal U,\theta),
\qquad
\overline{P}^{*}(\boldsymbol f,Z,\varphi,\mathcal U,\theta)
=
\overline{P}(\boldsymbol f,Z,\varphi,\mathcal U,\theta).
\]
Consequently,
\[
\underline{P}(\boldsymbol f,Z,\varphi,\theta)
=
\lim_{|\mathcal U|\to 0}
\underline{P}^{*}(\boldsymbol f,Z,\varphi,\mathcal U,\theta),
\qquad
\overline{P}(\boldsymbol f,Z,\varphi,\theta)
=
\lim_{|\mathcal U|\to 0}
\overline{P}^{*}(\boldsymbol f,Z,\varphi,\mathcal U,\theta).
\]
\end{theorem}

\begin{proof}
Fix \(\alpha\in\mathbb R\).
The case \(\theta=1\) is trivial, so we assume \(\theta\in(0,1)\).
Since the interval \([N,N/\theta]\) is contained in \([N,N/\theta+1)\), we have
\begin{equation}\label{eq:nds-window-1}
M(Z,\alpha,\varphi,\mathcal U,N,\theta)
\le
M^{*}(Z,\alpha,\varphi,\mathcal U,N,\theta).
\end{equation}
Set
$C(\alpha):=\#(\mathcal U)\,e^{|\alpha|+\|\varphi\|}.$
We claim that
\begin{equation}\label{eq:nds-window-2}
M^{*}(Z,\alpha,\varphi,\mathcal U,N+1,\theta)
\le
C(\alpha)\,
M(Z,\alpha,\varphi,\mathcal U,N,\theta).
\end{equation}
To prove \eqref{eq:nds-window-2}, take any
\[
\mathcal G\subseteq \bigcup_{N\le m< N/\theta+1}\mathcal S_m(\mathcal U)
\]
covering \(Z\). Write
$\mathcal G=\mathcal G_0\cup \mathcal G_1,$
where
\[
\mathcal G_0:=\{\mathbf W\in\mathcal G:\ m(\mathbf W)=N\},
\qquad
\mathcal G_1:=\{\mathbf W\in\mathcal G:\ m(\mathbf W)>N\}.
\]
For each \(\mathbf W\in\mathcal G_0\), consider its one-step extensions
\[
\mathcal E(\mathbf W):=\{\mathbf WU:\ U\in\mathcal U\}\subset \mathcal S_{N+1}(\mathcal U),
\]
and define
\[
\mathcal H:=\mathcal G_1\ \cup\ \bigcup_{\mathbf W\in\mathcal G_0}\mathcal E(\mathbf W).
\]
Then \(\mathcal H\) still covers \(Z\), since for every \(\mathbf W\in\mathcal S_N(\mathcal U)\),
\[
X(\mathbf W)=\bigcup_{U\in\mathcal U}X(\mathbf WU).
\]
Moreover, every \(\mathbf V\in\mathcal H\) satisfies
\[
N+1\le m(\mathbf V)<\frac{N}{\theta}+1
\le \frac{N+1}{\theta}.
\]
Hence \(\mathcal H\) is admissible for
\(M^{*}(Z,\alpha,\varphi,\mathcal U,N+1,\theta)\).

Now let \(\mathbf W\in\mathcal G_0\) and \(U\in\mathcal U\). Since \(X(\mathbf WU)\subseteq X(\mathbf W)\),
for every \(x\in X(\mathbf WU)\) we have
\[
S_{N+1}\varphi(x)
=
S_N\varphi(x)+\varphi(f_1^N x)
\le
S_N\varphi(x)+\|\varphi\|.
\]
Therefore,
\[
\sup_{x\in X(\mathbf WU)}S_{N+1}\varphi(x)
\le
\sup_{x\in X(\mathbf W)}S_N\varphi(x)+\|\varphi\|.
\]
It follows that
\[
\exp\left(
-\alpha(N+1)+\sup_{x\in X(\mathbf WU)}S_{N+1}\varphi(x)
\right)
\le
e^{|\alpha|+\|\varphi\|}
\exp\left(
-\alpha N+\sup_{x\in X(\mathbf W)}S_N\varphi(x)
\right).
\]
Summing over \(U\in\mathcal U\), we obtain
\[
\sum_{U\in\mathcal U}
\exp\left(
-\alpha(N+1)+\sup_{x\in X(\mathbf WU)}S_{N+1}\varphi(x)
\right)
\le
C(\alpha)\,
\exp\left(
-\alpha N+\sup_{x\in X(\mathbf W)}S_N\varphi(x)
\right).
\]
Consequently,
\begin{align*}
&\sum_{\mathbf V\in\mathcal H}
\exp\Bigl(
-\alpha\,m(\mathbf V)
+\sup_{x\in X(\mathbf V)}S_{m(\mathbf V)}\varphi(x)
\Bigr) \\
&\le
C(\alpha)
\sum_{\mathbf W\in\mathcal G}
\exp\Bigl(
-\alpha\,m(\mathbf W)
+\sup_{x\in X(\mathbf W)}S_{m(\mathbf W)}\varphi(x)
\Bigr).
\end{align*}
Taking the infimum over all such \(\mathcal G\) proves \eqref{eq:nds-window-2}.
Combining \eqref{eq:nds-window-1} and \eqref{eq:nds-window-2}, we obtain
\[
M(Z,\alpha,\varphi,\mathcal U,N+1,\theta)
\le
M^{*}(Z,\alpha,\varphi,\mathcal U,N+1,\theta)
\le
C(\alpha)\,M(Z,\alpha,\varphi,\mathcal U,N,\theta).
\]
Taking \(\liminf_{N\to\infty}\), we get
\[
\underline m(Z,\alpha,\varphi,\mathcal U,\theta)
\le
\underline m^{*}(Z,\alpha,\varphi,\mathcal U,\theta)
\le
C(\alpha)\,\underline m(Z,\alpha,\varphi,\mathcal U,\theta),
\]
and similarly,
\[
\overline m(Z,\alpha,\varphi,\mathcal U,\theta)
\le
\overline m^{*}(Z,\alpha,\varphi,\mathcal U,\theta)
\le
C(\alpha)\,\overline m(Z,\alpha,\varphi,\mathcal U,\theta).
\]
It follows that
\[
\underline{P}^{*}(\boldsymbol f, Z,\varphi,\mathcal U,\theta)
=
\underline{P}(\boldsymbol f, Z,\varphi,\mathcal U,\theta),
\qquad
\overline{P}^{*}(\boldsymbol f,Z,\varphi,\mathcal U,\theta)
=
\overline{P}(\boldsymbol f,Z,\varphi,\mathcal U,\theta).
\]
Finally, letting \(|\mathcal U|\to 0\), we obtain
\[
\underline{P}(\boldsymbol f,Z,\varphi,\theta)
=
\lim_{|\mathcal U|\to 0}
\underline{P}^{*}(\boldsymbol f,Z,\varphi,\mathcal U,\theta),
\qquad
\overline{P}(\boldsymbol f,Z,\varphi,\theta)
=
\lim_{|\mathcal U|\to 0}
\overline{P}^{*}(\boldsymbol f,Z,\varphi,\mathcal U,\theta).
\]
\end{proof}

\begin{proposition}\label{prop:basic-properties-all}
Let $(X,\boldsymbol{f})$ be an NDS, $Z\subseteq X$ a nonempty set and $\mathcal{P}\in\{\underline{P},\overline{P}\}$.
Then for any $\varphi,\psi\in C(X,\mathbb{R})$ and $\theta\in[0,1]$, the following properties hold:

\begin{itemize}

\item[\textnormal{(1)}]
$\underline{P}(\boldsymbol{f},Z,0,\theta)=\underline{h}_{\mathrm{top}}(\boldsymbol{f},Z,\theta)$, $\overline{P}(\boldsymbol{f},Z,0,\theta)=\overline{h}_{\mathrm{top}}(\boldsymbol{f},Z,\theta)$.

\item[\textnormal{(2)}]
\(
\underline{P}(\boldsymbol{f},Z,\varphi,\theta)
\le
\overline{P}(\boldsymbol{f},Z,\varphi,\theta).
\)

\item[\textnormal{(3)}] 
If $Z_1\subseteq Z_2 \subseteq X$, then 
\(
\mathcal{P}(\boldsymbol{f},Z_1,\varphi,\theta)
\le
\mathcal{P}(\boldsymbol{f},Z_2,\varphi,\theta).
\)

\item[\textnormal{(4)}] 
If $Z=\bigcup_{i\ge1}Z_i$, then 
\[
\mathcal{P}(\boldsymbol{f},Z,\varphi,\theta)
\ge
\sup_{i\ge1}
\mathcal{P}(\boldsymbol{f},Z_i,\varphi,\theta).
\]

\item[\textnormal{(5)}]
For any $Z_1,Z_2\subseteq X$,
\[
\overline{P}(\boldsymbol{f},Z_1\cup Z_2,\varphi,\theta)
=
\max\!\Bigl\{
\overline{P}(\boldsymbol{f},Z_1,\varphi,\theta),
\overline{P}(\boldsymbol{f},Z_2,\varphi,\theta)
\Bigr\}.
\]

\item[\textnormal{(6)}] 
If $0\le\theta<\phi\le1$, then
\(
\mathcal{P}(\boldsymbol{f},Z,\varphi,\theta)
\le
\mathcal{P}(\boldsymbol{f},Z,\varphi,\phi).
\)

\item[\textnormal{(7)}] 
For any \(c\in\mathbb{R}\),
\(
\mathcal{P}(\boldsymbol{f},Z,\varphi+c,\theta)
=
\mathcal{P}(\boldsymbol{f},Z,\varphi,\theta)+c.
\)

\item[\textnormal{(8)}]
If $\varphi\le\psi$, then 
\(
\mathcal{P}(\boldsymbol{f},Z,\varphi,\theta)
\le
\mathcal{P}(\boldsymbol{f},Z,\psi,\theta).
\)
In particular,
\[
\underline{h}_{\mathrm{top}}(\boldsymbol{f},Z,\theta)+\inf\varphi
\le
\underline{P}(\boldsymbol{f},Z,\varphi,\theta)
\le
\underline{h}_{\mathrm{top}}(\boldsymbol{f},Z,\theta)+\sup\varphi,
\]
\[
\overline{h}_{\mathrm{top}}(\boldsymbol{f},Z,\theta)+\inf\varphi
\le
\overline{P}(\boldsymbol{f},Z,\varphi,\theta)
\le
\overline{h}_{\mathrm{top}}(\boldsymbol{f},Z,\theta)+\sup\varphi.
\]
Moreover, $\mathcal{P}(\boldsymbol{f},Z,\cdot,\theta)$ is either finite-valued or identically $+\infty$.

\item[\textnormal{(9)}] 
For every finite open cover $\mathcal U$ of $X$,
\(
\bigl|
\mathcal{P}(\boldsymbol{f},Z,\varphi,\mathcal{U},\theta)
-
\mathcal{P}(\boldsymbol{f},Z,\psi,\mathcal{U},\theta)
\bigr|
\le
\|\varphi-\psi\|,
\) and so if \(\mathcal{P}(\boldsymbol{f},Z,\cdot,\theta) < \infty\), then \(
\bigl|
\mathcal{P}(\boldsymbol{f},Z,\varphi,\theta)
-
\mathcal{P}(\boldsymbol{f},Z,\psi,\theta)
\bigr|
\le
\|\varphi-\psi\|.
\)  
In other words, \(\mathcal{P}(\boldsymbol{f},Z,\cdot,\theta)\) is a continuous function on \(C(X,\mathbb{R})\).

\item[\textnormal{(10)}] 
\(
\mathcal{P}(\boldsymbol{f},Z,c\varphi,\theta)
\begin{cases}
\le c\,\mathcal{P}(\boldsymbol{f},Z,\varphi,\theta), & \text{if } c\ge1,\\[2mm]
\ge c\,\mathcal{P}(\boldsymbol{f},Z,\varphi,\theta), & \text{if } 0 < c\le1.
\end{cases}
\)

% \item[\textnormal{(10)}] 
% \(
% \bigl|
% \mathcal{P}(\boldsymbol{f},Z,\varphi,\theta)
% \bigr|
% \le
% \mathcal{P}(\boldsymbol{f},Z,|\varphi|,\theta).
% \)

\end{itemize}
\end{proposition}

\begin{proof}
\textnormal{(1)}--\textnormal{(8)} can be verified directly from the definitions.

(9) 
Since for every $x\in X$ and $n\in\mathbb N$ we have
\[
\bigl|S_n\varphi(x)-S_n\psi(x)\bigr|
=\biggl|\sum_{k=0}^{n-1}\bigl(\varphi(f_1^k x)-\psi(f_1^k x)\bigr)\biggr|
\le n\|\varphi-\psi\|,
\]
then \(
S_n\varphi(x)\le S_n\psi(x)+n\|\varphi-\psi\|.\)
For any cover $\Gamma$ of $Z$ with $\Gamma \subseteq \bigcup_{N \leq m < N/\theta+1} \mathcal{S}_{m}(\mathcal{U})$, we obtain
\[
\begin{aligned}
&\sum_{\mathbf U \in \Gamma}
\exp\!\left(
-\alpha m(\mathbf U)
+\sup_{x\in X(\mathbf U)} S_{m(\mathbf U)}\varphi(x)
\right)
\\[2mm]
&\qquad\le\;
\sum_{\mathbf U\in\Gamma}
\exp\!\left(
-(\alpha-\|\varphi-\psi\|) m(\mathbf U)
+\sup_{x\in X(\mathbf U)} S_{m(\mathbf U)}\psi(x)
\right).
\end{aligned}
\]
Taking the infimum over all such covers $\Gamma$ gives
\[
M(Z,\alpha,\varphi,\mathcal U,N,\theta)
\le
M(Z,\alpha-\|\varphi-\psi\|,\psi,\mathcal U,N,\theta).
\]
Taking $\limsup$ and $\liminf$ over $N$ respectively, we obtain
\[
\overline m(Z,\alpha,\varphi,\mathcal U,\theta)
\le
\overline m(Z,\alpha-\|\varphi-\psi\|,\psi,\mathcal U,\theta),
\]
\[
\underline m(Z,\alpha,\varphi,\mathcal U,\theta)
\le
\underline m(Z,\alpha-\|\varphi-\psi\|,\psi,\mathcal U,\theta),
\]
which implies
\[
\mathcal P(\boldsymbol f,Z,\varphi,\mathcal U,\theta)
\le
\mathcal P(\boldsymbol f,Z,\psi,\mathcal U,\theta)+\|\varphi-\psi\|.
\]
Interchanging the roles of $\varphi$ and $\psi$ we also get
\[
\mathcal P(\boldsymbol f,Z,\psi,\mathcal U,\theta)
\le
\mathcal P(\boldsymbol f,Z,\varphi,\mathcal U,\theta)+\|\varphi-\psi\|.
\]
Combining these two inequalities gives
\[
\bigl|
\mathcal{P}(\boldsymbol{f},Z,\varphi,\mathcal{U},\theta)
-
\mathcal{P}(\boldsymbol{f},Z,\psi,\mathcal{U},\theta)
\bigr|
\le
\|\varphi-\psi\|,
\]
which is the desired estimate. If \(\overline{P}(\boldsymbol{f},Z,\cdot,\theta) < \infty\), letting $|\mathcal U|\to0$, we obtain
\[
\bigl|
\mathcal{P}(\boldsymbol{f},Z,\varphi,\theta)
-
\mathcal{P}(\boldsymbol{f},Z,\psi,\theta)
\bigr|
\le
\|\varphi-\psi\|.
\]

(10) Fix an arbitrary finite open cover $\mathcal U$ of $X$.
If \(c\ge 1\), then for every \(s>c\,\overline{P}(\boldsymbol{f},Z,\varphi,\mathcal{U},\theta)\), 
\[
\overline{m}\left(Z,\frac{s}{c},\varphi,\mathcal{U},\theta \right)=0.
\]
Hence for every \(\varepsilon\in(0,1]\), there exists 
$N_0 \in \mathbb{N}$ such that for all $N \ge N_0$, 
\[
M\left(Z,\frac{s}{c},\varphi,\mathcal{U},N,\theta\right)<\varepsilon.
\]
Thus one can choose a family
\(
\mathcal{G}
\subseteq
\bigcup_{N \le m < N/\theta + 1}\mathcal{S}_m(\mathcal{U})
\)
covering $Z$ and satisfying
\[
\sum_{\mathbf{U}\in\mathcal{G}}
\exp\!\left(
-\frac{s}{c} m(\mathbf{U})
+ \sup_{x\in X(\mathbf{U})} S_{m(\mathbf{U})}\varphi(x)
\right)
<\varepsilon \le 1.
\]
It follows that
\[
\begin{aligned}
&\sum_{\mathbf U \in \mathcal{G}}
\exp\!\left(
   -s\, m(\mathbf U)
   +\sup_{x\in X(\mathbf U)} S_{m(\mathbf U)}(c\varphi)(x)
 \right)  \\
= &
\sum_{\mathbf U \in \mathcal{G}}
\left[
  \exp\!\left(
    -\tfrac{s}{c}\, m(\mathbf U)
    +\sup_{x\in X(\mathbf U)} S_{m(\mathbf U)}\varphi(x)
  \right)
\right]^{c}\\
\le & \sum_{\mathbf U \in \mathcal{G}}
\exp\!\left(
-\frac{s}{c} m(\mathbf U)
+\sup_{x\in X(\mathbf U)} S_{m(\mathbf U)}\varphi(x)
\right),
\end{aligned}
\]
which implies that 
\[
M(Z,s,c\varphi,\mathcal{U},N,\theta)\le M\left(Z,\frac{s}{c},\varphi,\mathcal{U},N,\theta\right)<\varepsilon
\]
for all \(N \ge N_0\), and therefore
\[
\overline{m}(Z,s,c\varphi,\mathcal U,\theta)=0.
\]
By the definition of the upper pressure,
\[
\overline{P}(\boldsymbol{f},Z,c\varphi,\mathcal{U},\theta)
\le s.
\]
Since $s > c\,\overline{P}(\boldsymbol{f},Z,\varphi,\mathcal{U},\theta)$ is arbitrary,
\[
\overline{P}(\boldsymbol{f},Z,c\varphi,\mathcal{U},\theta)
\le
c\,\overline{P}(\boldsymbol{f},Z,\varphi,\mathcal{U},\theta).
\]
Taking the limit
$|\mathcal{U}|\to 0$ gives
\[
\overline{P}(\boldsymbol{f},Z,c\varphi,\theta)
\le
c\,\overline{P}(\boldsymbol{f},Z,\varphi,\theta).
\]
Similarly, if \(0 < c \le 1\), we have 
\[
\overline{P}(\boldsymbol{f},Z,c\varphi,\theta)
\ge
c\,\overline{P}(\boldsymbol{f},Z,\varphi,\theta).
\]
The same argument with $\liminf$ in place of $\limsup$ yields the corresponding estimate for $\underline{P}$.

\end{proof}

\subsection{Continuity with respect to $\theta$}
In our previous work~\cite{ju2026entropy}, we established the continuity of the lower and upper 
\(\theta\)-intermediate topological entropies for \(\theta \in (0,1]\) and presented an example 
demonstrating possible discontinuities at \(\theta = 0\). 
This phenomenon closely parallels that of the intermediate dimensions introduced 
by Falconer~\cite{Falconer2021Intermediate}, where for \(0 < \theta < \phi \le 1\), the upper and lower \(\theta\)-intermediate dimensions satisfy
\[
\overline{\dim}_{\theta} F
\;\le\;
\overline{\dim}_{\phi} F
\;\le\;
\frac{\phi}{\theta}\,\overline{\dim}_{\theta} F,
\qquad
\underline{\dim}_{\theta} F
\;\le\;
\underline{\dim}_{\phi} F
\;\le\;
\frac{\phi}{\theta}\,\underline{\dim}_{\theta} F.
\]

In what follows, we extend these continuity and comparison results to the lower and upper 
\(\theta\)-intermediate topological pressures.

\begin{proposition}
Let $(X,d)$ be a compact metric space, $\boldsymbol{f}$ a sequence of continuous self-maps of $X$ and $\varphi\in C(X,\mathbb{R})$. For any nonempty \(Z\subseteq X\) and \(0<\theta<\phi\le1\), we have
\[
\overline{P}(\boldsymbol{f},Z,\varphi,\theta)
\;\le\;
\overline{P}(\boldsymbol{f},Z,\varphi,\phi)
\;\le\;
\frac{\phi}{\theta}\,\overline{P}(\boldsymbol{f},Z,\varphi,\theta)+\Bigl(\frac{\phi}{\theta}-1\Bigr)\|\varphi\|,
\]
\[
\underline{P}(\boldsymbol{f},Z,\varphi,\theta)
\;\le\;
\underline{P}(\boldsymbol{f},Z,\varphi,\phi)
\;\le\;
\frac{\phi}{\theta}\,\underline{P}(\boldsymbol{f},Z,\varphi,\theta)+\Bigl(\frac{\phi}{\theta}-1\Bigr)\|\varphi\|.
\]
In particular, when $\varphi \equiv  0$, these inequalities reduce to the corresponding ones for the 
$\theta$-intermediate topological entropies:
\[
\overline{h}_{\mathrm{top}}(\boldsymbol{f},Z,\theta)
\;\le\;
\overline{h}_{\mathrm{top}}(\boldsymbol{f},Z,\phi)
\;\le\;
\frac{\phi}{\theta}\,
\overline{h}_{\mathrm{top}}(\boldsymbol{f},Z,\theta),
\]
\[
\underline{h}_{\mathrm{top}}(\boldsymbol{f},Z,\theta)
\;\le\;
\underline{h}_{\mathrm{top}}(\boldsymbol{f},Z,\phi)
\;\le\;
\frac{\phi}{\theta}\,
\underline{h}_{\mathrm{top}}(\boldsymbol{f},Z,\theta),
\]
which were established in~\cite{ju2026entropy}.
\end{proposition}

\begin{proof}
The left-hand inequality follows from the monotonicity of 
\(\overline{P}(\boldsymbol{f},Z,\varphi,\theta)\) in~\(\theta\).
To prove the right-hand inequality, fix \(0<\theta<\phi\le1\) and a finite open cover \(\mathcal U\) of \(X\). 
Set \(\delta=|\mathcal U|\) and \(M=\|\varphi\|\). 
Fix \(s>\overline{P}^{*}(\boldsymbol{f},Z,\varphi,\mathcal U,\theta)\) and \(\varepsilon>0\). Using the same comparison argument as in Proposition~\ref{prop:basic-properties-all}\textnormal{(8)}, we see that 
$\overline{P}^{*}(\boldsymbol f,Z,\varphi,\mathcal U,\theta)\ge -M$,
and therefore $s+M>0$.
Now, by the definition of \(\overline{P}^{*}(\boldsymbol{f},Z,\varphi,\mathcal U,\theta)\), there exists \(N_0\) such that for every \(N\ge N_0\) one can choose a family
$\mathcal G\subseteq\bigcup_{N\le p \le N/\theta}\mathcal S_p(\mathcal U)$
covering \(Z\) and satisfying
\[
\sum_{\mathbf U\in\mathcal G}\exp\biggl(-s\,m(\mathbf U)+\sup_{x \in X(\mathbf{U})}S_{m(\mathbf{U})}\varphi(x)\biggr)<\varepsilon.
\]
Split \(\mathcal G=\mathcal G_0\cup\mathcal G_1\), where
\[
\mathcal G_0=\left\{\mathbf U\in\mathcal G:N\le m(\mathbf U)\le N/\phi\right\},\qquad
\mathcal G_1=\left\{\mathbf U\in\mathcal G:N/\phi < m(\mathbf U)\le N/\theta\right\}.
\]
Set \(q=\lfloor N/\phi\rfloor\), where \( \lfloor x \rfloor \) denotes the greatest integer less than or equal to \(x\) and define
\[
\mathcal{G}_{1}^{*}:=\left\{\mathbf{U}^{(\phi)}=\mathbf{U}|_{[0,q-1]}:\mathbf{U}\in\mathcal{G}_{1}\right\}
   \subseteq\mathcal{S}_q(\mathcal{U}).
\]
By construction, \(X(\mathbf U)\subseteq X(\mathbf U^{(\phi)})\), 
hence \(\mathcal G_0\cup\mathcal G_1^*\) still covers \(Z\).
Define
\[
t_N
=\frac{s}{q}\cdot\frac{N}{\theta}
+\left(\frac{N}{\theta q}-1\right)M
+\omega(\delta)
\]
so that
\[
t_N\longrightarrow 
\frac{\phi}{\theta}s+\Bigl(\frac{\phi}{\theta}-1\Bigr)M+\omega(\delta)
\quad\text{as }N\to\infty.
\]
For each \(\mathbf U\in\mathcal G_1\) and its prefix 
\(\mathbf U^{(\phi)}\) of length \(q\), one has
\[
\sup_{y \in X\left(\mathbf U^{(\phi)}\right)}S_q \varphi(y) \le \sup_{x \in X\left(\mathbf U\right)}S_{m(\mathbf U)}\varphi(x)+ (m(\mathbf U)-q)M + q\,\omega(\delta).
\]
For such \(\mathbf U\in\mathcal G_1\),
\begin{align*}
&\exp\biggl(-t_N q+\sup_{x \in X\left(\mathbf U^{(\phi)}\right)}S_q \varphi(x)\biggr)\\
= & \exp\biggl(-(s+M)\frac{N}{\theta}+qM-q\omega(\delta)+\sup_{x \in X\left(\mathbf U^{(\phi)}\right)}S_q \varphi(x)\biggr)\\
\le &
\exp\biggl(-s m(\mathbf U)-(m(\mathbf U)-q)M-q\omega(\delta)+\sup_{x \in X\left(\mathbf U^{(\phi)}\right)}S_q \varphi(x)\biggr)\\
\le & 
\exp\biggl(-s m(\mathbf U)+\sup_{x \in X\left(\mathbf U\right)}S_{m(\mathbf U)}\varphi(x)\biggr).
\end{align*}
For \(\mathbf U\in\mathcal G_0\), we have \(t_N\ge s\) for all \(N \in \mathbb{N}\), thus
\[
\exp\biggl(-t_N m(\mathbf U)+\sup_{x \in X\left(\mathbf U\right)}S_{m(\mathbf U)}\varphi(x)\biggr)
\le
\exp\biggl(-s m(\mathbf U)+\sup_{x \in X\left(\mathbf U\right)}S_{m(\mathbf U)}\varphi(x)\biggr).
\]
Summing over \(\mathcal G_0\cup\mathcal G_1^*\) yields
\[
\begin{aligned}
&\sum_{\mathbf V\in \mathcal G_0\cup\mathcal G_1^*}
  \exp\Bigl(
     -t_N m(\mathbf V)
     +\sup_{x \in X(\mathbf V)}
       S_{m(\mathbf V)}\varphi(x)
  \Bigr)
\\
 & \leq
 \sum_{\mathbf U\in\mathcal G}
   \exp\Bigl(
      -s m(\mathbf U)
      +\sup_{x \in X(\mathbf U)}
        S_{m(\mathbf U)}\varphi(x)
   \Bigr)
   <\varepsilon.
\end{aligned}
\]
Hence \(M^{*}(Z,t_N,\varphi,\mathcal U,N,\phi)<\varepsilon\) for all \(N \geq N_0\).
For any 
\[
t>\frac{\phi}{\theta}s+\Bigl(\frac{\phi}{\theta}-1\Bigr)M+\omega(\delta),
\]
there exists \(N_1\) such that \(t>t_N\) for all \(N\ge N_1\),
and therefore
\[
M^{*}(Z,t,\varphi,\mathcal U,N,\phi)\le M^{*}(Z,t_N,\varphi,\mathcal U,N,\phi)<\varepsilon
\quad\text{for all }N\ge\max\{N_0,N_1\}.
\]
Taking the upper limit as \(N\to\infty\) gives
\(\overline m^{*}(Z,t,\varphi,\mathcal U,\phi)=0\),
and hence
\[
\overline{P}^{*}(\boldsymbol{f},Z,\varphi,\mathcal U,\phi)\le t.
\]
Letting 
\(t\downarrow 
\frac{\phi}{\theta}s+(\frac{\phi}{\theta}-1)M+\omega(\delta)\)
and
\(s\downarrow \overline{P}^{*}(\boldsymbol{f},Z,\varphi,\mathcal U,\theta)\)
gives
\[
\overline{P}^{*}(\boldsymbol{f},Z,\varphi,\mathcal U,\phi)
\le
\frac{\phi}{\theta}\,
\overline{P}^{*}(\boldsymbol{f},Z,\varphi,\mathcal U,\theta)
+\Bigl(\frac{\phi}{\theta}-1\Bigr)M+\omega(\delta).
\]
Finally, letting \(|\mathcal U|\to0\) and noting that \(\omega(\delta)\to0\),
we obtain
\[
\overline{P}(\boldsymbol{f},Z,\varphi,\phi)
\le
\frac{\phi}{\theta}\,
\overline{P}(\boldsymbol{f},Z,\varphi,\theta)
+\Bigl(\frac{\phi}{\theta}-1\Bigr)M.
\]
The argument for the lower pressure is completely analogous.
\end{proof}

\begin{corollary}
The maps 
\(
\theta \mapsto 
\underline{P}(\boldsymbol{f},Z,\varphi,\theta)\)
and
\(\theta \mapsto 
\overline{P}(\boldsymbol{f},Z,\varphi,\theta)
\)
are continuous for \(\theta\in(0,1]\).
\end{corollary}

\subsection{Equivalent definition of pressures by Bowen balls}

We now give an equivalent definition of the lower and upper \(\theta\)-intermediate topological pressures by Bowen balls. 
For any $\alpha \in \mathbb{R}$, $N \in \mathbb{N}$, $\delta>0$, $\varphi \in C(X,\mathbb{R})$ and $\theta\in [0,1]$, define
\[
M(Z, \alpha,\varphi, \delta,N, \theta) 
= \inf \left\{\sum_i \exp\biggl(-\alpha\,n_i+\sup_{y \in B_{n_i}(x_i,\delta)}S_{n_i}\varphi(y)\biggr)\right\},
\]
where the infimum is taken over all finite or countable collections \(\mathcal{F}=\{B_{n_i}(x_i,\delta)\}_{i}\) such
that \(x_i\in X,~N \le n_i \le  N/\theta \) if $\theta>0$, and \(n_i \ge N\) if $\theta=0$, and \(\mathcal{F}\) covers \(Z\), i.e., \(Z \subseteq \bigcup_{i}B_{n_i}(x_i,\delta)\).

Let
\[\underline{m}(Z, \alpha,\varphi, \delta,\theta)=\liminf_{N \rightarrow \infty} M(Z, \alpha, \varphi,\delta, N,\theta),\]
\[\overline{m}(Z, \alpha,\varphi, \delta,\theta)=\limsup_{N \rightarrow \infty} M(Z, \alpha, \varphi,\delta, N,\theta).\]

We define the lower and upper \(\theta\)-intermediate topological pressures of \(Z\) relative to \(\delta\) by
\[\underline{P}(\boldsymbol{f},Z,\varphi,\delta,\theta)=\inf \left\{\alpha: \underline{m}(Z, \alpha,\varphi,\delta,\theta)=0\right\}=\sup \left\{\alpha: \underline{m}(Z, \alpha, \varphi,\delta,\theta)=\infty\right\},
\]
\[\overline{P}(\boldsymbol{f},Z,\varphi,\delta,\theta)=\inf \{\alpha: \overline{m}(Z, \alpha, \varphi,\delta,\theta)=0\}=\sup \{\alpha: \overline{m}(Z, \alpha, \varphi,\delta,\theta)=\infty\}.\]

\begin{theorem}
For any set \(Z \subseteq X\),  \(\varphi \in C(X,\mathbb{R})\) and \(\theta \in [0,1]\), the following limits exist:
\[
\underline{P}(\boldsymbol{f},Z,\varphi,\theta)
=\lim_{\delta \to 0}\underline{P}(\boldsymbol{f},Z,\varphi,\delta,\theta),
\qquad
\overline{P}(\boldsymbol{f},Z,\varphi,\theta)
=\lim_{\delta \to 0}\overline{P}(\boldsymbol{f},Z,\varphi,\delta,\theta).
\]
% \begin{proof}
% Let \(\mathcal{U}\) be a finite open cover of \(X\) and \(\delta(\mathcal{U})\) its Lebesgue number. 
% It is easy to see that for every \(x \in X,\) if \(x \in X(\mathbf{U})\) for some 
% \(\mathbf{U} \in \mathcal{S}_{n}(\mathcal{U})\), then
% \[
% B_{n}\!\left(x,\tfrac{1}{2}\delta(\mathcal{U})\right) \subseteq 
% X(\mathbf{U}) \subseteq 
% B_{n}(x,2|\mathcal{U}|).
% \]
% Thus,
% \[
% M(Z,\alpha,\varphi,2|\mathcal{U}|,N,\theta)
% \le 
% M(Z,\alpha,\varphi,\mathcal{U},N,\theta)
% \le 
% M\!\left(Z,\alpha,\varphi,\tfrac{1}{2}\delta(\mathcal{U}),N,\theta\right).
% \]
% This implies
% \[
% \underline{P}\!\left(\boldsymbol{f},Z,\varphi,2|\mathcal{U}|,\theta\right)
% \le
% \underline{P}\!\left(\boldsymbol{f},Z,\varphi,\mathcal{U},\theta\right)
% \le
% \underline{P}\!\left(\boldsymbol{f},Z,\varphi,\tfrac{1}{2}\delta(\mathcal{U}),\theta\right),
% \]
% and similarly for \(\overline{P}\). 
% Letting \(|\mathcal{U}|\to0\) (hence \(\delta(\mathcal{U})\to0\)) yields the desired limits.
% \end{proof}
\begin{proof}
The proof is analogous to that of \cite[Theorem 2.11]{ju2026entropy} and is therefore omitted.
\end{proof}
\end{theorem}

If we replace \(\sup_{y \in B_{n_i}(x_i,\delta)} S_{n_i}\varphi(y)\) in the definition of \(M(Z,\alpha,\varphi,\delta,N,\theta)\) by \(S_{n_i}\varphi(x_i)\), then we can define new functions 
\(\mathcal{M}\), \(\underline{\mathfrak{m}}\) and \(\overline{\mathfrak{m}}\).
For any set \(Z \subseteq X\) and \(\delta > 0\),
we denote the respective critical values by
\[
\underline{P}'(\boldsymbol{f},Z,\varphi,\delta,\theta)\quad
\text{and}\quad
\overline{P}'(\boldsymbol{f},Z,\varphi,\delta,\theta).
\]

\begin{proposition}
\label{prop:pressure_equivalence}
Let \((X,d)\) be a compact metric space, \(\boldsymbol{f}\) a sequence of continuous self-maps of \(X\) and \(\varphi \in C(X,\mathbb{R})\).
For any \(Z \subseteq X\) and \(\theta \in [0,1]\), we have
\[
\underline{P}(\boldsymbol{f},Z,\varphi,\theta)
=\lim_{\delta \to 0}\underline{P}'(\boldsymbol{f},Z,\varphi,\delta,\theta),
\qquad
\overline{P}(\boldsymbol{f},Z,\varphi,\theta)
=\lim_{\delta \to 0}\overline{P}'(\boldsymbol{f},Z,\varphi,\delta,\theta).
\]
\end{proposition}
\begin{proof}
We follow the idea given in \cite{zhong2023variational}. Fix \(\delta>0\).
It is clear that 
\[
\overline{P}'(\boldsymbol{f},Z,\varphi,\delta,\theta)
\le \overline{P}(\boldsymbol{f},Z,\varphi,\delta,\theta).
\]
Then for any \(x\in X\) and \(n\in\mathbb{N}\),
\[
S_n\varphi(x)
\le
\sup_{y\in B_n(x,\delta)}S_n\varphi(y)
\le
S_n\varphi(x)+n\omega(\delta).
\]
Hence
\[
\begin{aligned}
\mathcal{M}(Z,\alpha,\varphi,\delta,N,\theta)
&=\inf\Bigl\{ \sum_i e^{-\alpha n_i+S_{n_i}\varphi(x_i)}\Bigr\} \\
&\ge \inf\Bigl\{ \sum_i e^{-(\alpha+\omega(\delta))n_i+ \sup_{y\in B_{n_i}(x_i,\delta)}S_{n_i}\varphi(y)}\Bigr\} \\
&=M\left(Z,\alpha+\omega(\delta),\varphi,\delta,N,\theta\right).
\end{aligned}
\]
Taking the $\limsup_{N\to\infty}$ on both sides yields
\[
\overline{\mathfrak{m}}(Z,\alpha,\varphi,\delta,\theta)
\ge
\overline{m}(Z,\alpha+\omega(\delta),\varphi,\delta,\theta).
\]
This implies that
\[
\overline{P}(\boldsymbol{f},Z,\varphi,\delta,\theta)
\le
\overline{P}'(\boldsymbol{f},Z,\varphi,\delta,\theta)+\omega(\delta).
\]
It then follows that
\[
\overline{P}'(\boldsymbol{f},Z,\varphi,\delta,\theta)
\le
\overline{P}(\boldsymbol{f},Z,\varphi,\delta,\theta)
\le
\overline{P}'(\boldsymbol{f},Z,\varphi,\delta,\theta)+\omega(\delta),
\]
and the same inequality holds for \(\underline{P}\).
Since \(\omega(\delta)\to0\) as \(\delta\to0\), the desired equalities follow by letting \(\delta\to0\).
\end{proof}

\section{Dynamical properties of intermediate topological pressures}\label{sec:properties}

In this section, we establish several fundamental properties of the intermediate topological pressures for $(X,\boldsymbol{f})$. 
These results are inspired by the corresponding properties of classical topological pressure for NDSs in \cite{kong2015topological}, and extend them to the intermediate topological pressures.

\begin{proposition}\label{sc}
Let \((X,d)\) be a compact metric space and \( \boldsymbol{f} \) be a sequence of continuous self-maps of \( X \). 
Then, for any \(Z\subseteq X, \varphi \in C(X,\mathbb{R})\) and \(\theta\in(0,1]\),
\[
\underline{P}(\boldsymbol{f},\overline Z,\varphi,\theta)
=
\underline{P}(\boldsymbol{f},Z,\varphi,\theta),
\quad
\overline{P}(\boldsymbol{f},\overline Z,\varphi,\theta)
=
\overline{P}(\boldsymbol{f},Z,\varphi,\theta).
\]
\end{proposition}

\begin{proof}
The proof is identical to that of Proposition 3.1 in \cite{ju2026entropy}, 
so we omit the details.
\end{proof}

\begin{proposition}
Let \((X,d)\) be a compact metric space, \( \boldsymbol{f} \) a sequence of equicontinuous self-maps of $X$, and $\varphi \equiv c$ a constant function. Then, for any $\theta\in [0,1]$ and all $m \in \mathbb{N}$,
\[
\underline P\left(\boldsymbol{f}^m,Z,S_m\varphi,\theta\right)
= m\,\underline{P}\left(\boldsymbol{f},Z,\varphi,\theta\right),
\qquad
\overline P\left(\boldsymbol{f}^m,Z,S_m\varphi,\theta\right)
= m\,\overline{P}\left(\boldsymbol{f},Z,\varphi,\theta\right).
\]
\end{proposition}

\begin{proof}
By Proposition 3.4 in \cite{ju2026entropy}, we have
\[
\underline{h}_{\mathrm{top}}(\boldsymbol{f}^m,Z,\theta)
= m\,\underline{h}_{\mathrm{top}}(\boldsymbol{f},Z,\theta),\quad
\overline{h}_{\mathrm{top}}(\boldsymbol{f}^m,Z,\theta)
= m\,\overline{h}_{\mathrm{top}}(\boldsymbol{f},Z,\theta).
\]
Since $\varphi \equiv c$, and according to (7) of Proposition \ref{prop:basic-properties-all},
\[
\underline{P}(\boldsymbol{f}^m,Z,S_m\varphi,\theta)
= mc + \underline{h}_{\mathrm{top}}(\boldsymbol{f}^m,Z,\theta)
= m\bigl(c + \underline{h}_{\mathrm{top}}(\boldsymbol{f},Z,\theta)\bigr)
= m\,\underline{P}(\boldsymbol{f},Z,\varphi,\theta),
\]
and similarly for the upper pressure.
\end{proof}

\begin{proposition}\label{pr-theta}
Let \( (X,d) \) be a compact metric space and \( \boldsymbol{f} \) be a periodic sequence
of continuous self-maps of \(X\) with period \(m\), that is,
$f_{n+m}=f_n,~n \in \mathbb{N}.$
Then for any \(Z\subseteq X\), \(\varphi\in C(X,\mathbb R)\) and \(\theta\in[0,1]\),
\[
\underline{P}\left(\boldsymbol{f}^m,Z,S_m\varphi,\theta\right)
= m\,\underline{P}\left(\boldsymbol{f},Z,\varphi,\theta\right),
\qquad
\overline{P}\left(\boldsymbol{f}^m,Z,S_m\varphi,\theta\right)
= m\,\overline{P}\left(\boldsymbol{f},Z,\varphi,\theta\right).
\]
\end{proposition}

\begin{proof}
The case $\theta=0$ has been established in \cite[Theorem~5.12]{chen2025nonautonomous}. 
Therefore, it suffices to consider the case $\theta\in(0,1]$.
The proof is divided into two parts. 

\noindent\textbf{Part I.}
We prove that
\[
\underline{P}(\boldsymbol f^m,Z,S_m\varphi,\theta)
\le
m\,\underline{P}(\boldsymbol f,Z,\varphi,\theta),
\quad
\overline{P}(\boldsymbol f^m,Z,S_m\varphi,\theta)
\le
m\,\overline{P}(\boldsymbol f,Z,\varphi,\theta).
\]

Fix \(\delta>0\) and \(\alpha \in \mathbb{R}\), and write \(M=\|\varphi\|\).
Let \(k=mN+r\), where \(N\in\mathbb N\) and \(0\le r<m\).
Take any cover
$\left\{B_{n_i,\boldsymbol f}(x_i,\delta)\right\}_i$
of \(Z\), where \(n_i\in[k,k/\theta]\) for all \(i\).
The argument is carried out in three steps.

\emph{Step~1.}
Set \(t_i=n_i-r\). Since \(t_i\le n_i\), we have
$B_{n_i,\boldsymbol f}(x_i,\delta)\subseteq B_{t_i,\boldsymbol f}(x_i,\delta),$
so \(\{B_{t_i,\boldsymbol f}(x_i,\delta)\}_i\) is still a cover of \(Z\).
Moreover,
$t_i\in [mN,(mN+r)/\theta-r].$
Also,
\[
S_{n_i}^{\boldsymbol f}\varphi(x_i)
=
S_{t_i}^{\boldsymbol f}\varphi(x_i)
+\sum_{\ell=0}^{r-1}\varphi\bigl(f_1^{\,t_i+\ell}(x_i)\bigr)
\ge
S_{t_i}^{\boldsymbol f}\varphi(x_i)-rM.
\]
Hence
\begin{equation}\label{eq:step1-periodic}
\begin{aligned}
\sum_i e^{-\alpha n_i+S_{n_i}^{\boldsymbol f}\varphi(x_i)}
&=
\sum_i e^{-\alpha t_i-\alpha r+S_{n_i}^{\boldsymbol f}\varphi(x_i)} \\
&\ge
e^{-(|\alpha|+M)r}
\sum_i e^{-\alpha t_i+S_{t_i}^{\boldsymbol f}\varphi(x_i)}.
\end{aligned}
\end{equation}

\smallskip
\emph{Step~2.}
Set
\[
q:=\Bigl\lfloor\frac{mN}{\theta}\Bigr\rfloor,
\qquad
q_i:=\min\{q,t_i\},
\qquad
\Delta_i:=t_i-q_i\ge0.
\]
Then $q_i\in[mN,mN/\theta].$
Since \(q_i\le t_i\), we have
$
B_{t_i,\boldsymbol f}(x_i,\delta)\subseteq B_{q_i,\boldsymbol f}(x_i,\delta),
$
so \(\{B_{q_i,\boldsymbol f}(x_i,\delta)\}_i\) still covers \(Z\).
We next estimate \(\Delta_i\).
If \(t_i\le q\), then \(q_i=t_i\), and hence \(\Delta_i=0\).
If \(t_i>q\), then \(q_i=q\) and
\[
\Delta_i=t_i-q=n_i-r-q
\le \frac{k}{\theta}-r-\Bigl\lfloor\frac{mN}{\theta}\Bigr\rfloor.
\]
Since \(k=mN+r\) and \(\lfloor x\rfloor\ge x-1\), it follows that
\[
\Delta_i
\le 
\frac{mN+r}{\theta}-r-\Bigl(\frac{mN}{\theta}-1\Bigr)
=
\frac{r}{\theta}+1-r
<
\frac{m}{\theta}+1-r.
\]
Thus, in all cases,
\[
0\le \Delta_i<\frac{m}{\theta}+1-r.
\]
Furthermore,
\[
S_{t_i}^{\boldsymbol f}\varphi(x_i)
=
S_{q_i}^{\boldsymbol f}\varphi(x_i)
+\sum_{\ell=0}^{\Delta_i-1}\varphi\bigl(f_1^{\,q_i+\ell}(x_i)\bigr)
\ge
S_{q_i}^{\boldsymbol f}\varphi(x_i)-\Delta_i M.
\]
Therefore,
\begin{equation}\label{eq:step2-periodic}
\begin{aligned}
\sum_i e^{-\alpha t_i+S_{t_i}^{\boldsymbol f}\varphi(x_i)}
&\ge
\sum_i e^{-\alpha q_i-\alpha\Delta_i+S_{q_i}^{\boldsymbol f}\varphi(x_i)-\Delta_i M} \\
&=
\sum_i e^{-(\alpha+M)\Delta_i}e^{-\alpha q_i+S_{q_i}^{\boldsymbol f}\varphi(x_i)} \\
&\ge
e^{-(|\alpha|+M)(\frac{m}{\theta}+1-r)}
\sum_i e^{-\alpha q_i+S_{q_i}^{\boldsymbol f}\varphi(x_i)}.
\end{aligned}
\end{equation}

\smallskip
\emph{Step~3.}
Let
$p_i=\Bigl\lfloor\frac{q_i}{m}\Bigr\rfloor.$
Since \(q_i \in [mN,mN/\theta]\), we have
$p_i\in[N,N/\theta].$
For each \(i\), we also have
$B_{q_i,\boldsymbol f}(x_i,\delta)\subseteq B_{p_i,\boldsymbol f^m}(x_i,\delta),
$
hence \(\{B_{p_i,\boldsymbol f^m}(x_i,\delta)\}_i\) is a cover of \(Z\)
with respect to \(\boldsymbol f^m\).
Next,
\[
S_{q_i}^{\boldsymbol f}\varphi(x_i)
=
S_{mp_i}^{\boldsymbol f}\varphi(x_i)
+\sum_{\ell=0}^{q_i-mp_i-1}\varphi\bigl(f_1^{\,mp_i+\ell}(x_i)\bigr)
\ge
S_{mp_i}^{\boldsymbol f}\varphi(x_i)-mM,
\]
as \(0\le q_i-mp_i<m\).
Moreover, \(q_i<mp_i+m\) gives
$
e^{-\alpha q_i}\ge e^{-|\alpha|m}e^{-\alpha mp_i}.
$
Since \(\boldsymbol f\) has period \(m\), for every
\(p\in\mathbb N\) and \(x\in X\),
\begin{equation}\label{pr-eq}
 S_p^{\boldsymbol f^m}(S_m\varphi)(x)
=
\sum_{j=0}^{p-1}S_m\varphi\bigl(f_1^{\,jm}(x)\bigr)
=
\sum_{j=0}^{p-1}\sum_{\ell=0}^{m-1}\varphi\bigl(f_1^{\,jm+\ell}(x)\bigr)
=
S_{mp}^{\boldsymbol f}\varphi(x).   
\end{equation}
Therefore,
\begin{equation}\label{eq:step3-periodic}
\begin{aligned}
\sum_i e^{-\alpha q_i+S_{q_i}^{\boldsymbol f}\varphi(x_i)}
&\ge
\sum_i e^{-|\alpha|m-mM}\,
e^{-\alpha mp_i+S_{mp_i}^{\boldsymbol f}\varphi(x_i)} \\
&=
e^{-(|\alpha|+M)m}
\sum_i e^{-\alpha mp_i+S_{p_i}^{\boldsymbol f^m}(S_m\varphi)(x_i)}.
\end{aligned}
\end{equation}

Combining \eqref{eq:step1-periodic}, \eqref{eq:step2-periodic} and \eqref{eq:step3-periodic}, we obtain
\[
\sum_i e^{-\alpha n_i+S_{n_i}^{\boldsymbol f}\varphi(x_i)}
\ge
C
\sum_i e^{-\alpha mp_i+S_{p_i}^{\boldsymbol f^m}(S_m\varphi)(x_i)},
\]
where
$C=e^{-(|\alpha|+M)(\frac{m}{\theta}+m+1)}.$
Since \(\{B_{p_i,\boldsymbol f^m}(x_i,\delta)\}_i\) is an admissible cover of \(Z\) for
\(\mathcal M_{\boldsymbol f^m}(Z,\alpha m,S_m\varphi,\delta,N,\theta)\), it follows that
\[
\mathcal M_{\boldsymbol f}(Z,\alpha,\varphi,\delta,k,\theta)
\ge
C\,\mathcal M_{\boldsymbol f^m}(Z,\alpha m,S_m\varphi,\delta,N,\theta).
\]
Taking \(\liminf_{N\to\infty}\) and \(\limsup_{N\to\infty}\), respectively, we get
\[
\underline{\mathfrak m}_{\boldsymbol f}(Z,\alpha,\varphi,\delta,\theta)
\ge
C\,\underline{\mathfrak m}_{\boldsymbol f^m}(Z,\alpha m,S_m\varphi,\delta,\theta),
\]
and
\[
\overline{\mathfrak m}_{\boldsymbol f}(Z,\alpha,\varphi,\delta,\theta)
\ge
C\,\overline{\mathfrak m}_{\boldsymbol f^m}(Z,\alpha m,S_m\varphi,\delta,\theta).
\]
Hence,
\[
\underline P'\!\left(\boldsymbol f^m,Z,S_m\varphi,\delta,\theta\right)
\le
m\,\underline P'\!\left(\boldsymbol f,Z,\varphi,\delta,\theta\right),
\qquad
\overline P'\!\left(\boldsymbol f^m,Z,S_m\varphi,\delta,\theta\right)
\le
m\,\overline P'\!\left(\boldsymbol f,Z,\varphi,\delta,\theta\right).
\]
Finally, letting \(\delta\to0\), we obtain
\[
\underline P\!\left(\boldsymbol f^m,Z,S_m\varphi,\theta\right)
\le
m\,\underline P\!\left(\boldsymbol f,Z,\varphi,\theta\right),
\qquad
\overline P\!\left(\boldsymbol f^m,Z,S_m\varphi,\theta\right)
\le
m\,\overline P\!\left(\boldsymbol f,Z,\varphi,\theta\right).
\]
This completes Part~I.

\medskip
\noindent\textbf{Part~II.}
We prove that
\[
\underline{P}(\boldsymbol f^m,Z,S_m\varphi,\theta)
\ge
m\,\underline{P}(\boldsymbol f,Z,\varphi,\theta),
\quad
\overline{P}(\boldsymbol f^m,Z,S_m\varphi,\theta)
\ge
m\,\overline{P}(\boldsymbol f,Z,\varphi,\theta).
\]

For every $\eta>0$, define
\[
\delta = \delta(\eta):=\eta+
\max_{0\le j\le m-1}
\sup_{d(u,v)\le \eta}
d\bigl(f_1^j(u),f_1^j(v)\bigr).
\]
Since the maps $f_1^j$ ($0\le j\le m-1$) are uniformly continuous,
we have $\delta(\eta)\to 0$ as $\eta\to 0$. Moreover,
\[
d(u,v)<\eta
\quad\Longrightarrow\quad
d\bigl(f_1^j(u),f_1^j(v)\bigr) < \delta
\quad\text{for all }0\le j\le m-1.
\]

Now let
$\{B_{p_i,\boldsymbol f^m}(x_i,\eta)\}_i$
be any cover of \(Z\) with respect to \(\boldsymbol f^m\), where
$p_i\in[N,N/\theta]$.
We claim that for  every $i$,
$B_{p_i,\boldsymbol f^m}(x_i,\eta)\subseteq B_{mp_i,\boldsymbol f}(x_i,\delta).$
Indeed, if \(y\in B_{p_i,\boldsymbol f^m}(x_i,\eta)\), then for every \(0\le j<p_i\),
$d\bigl(f_1^{jm}(y),f_1^{jm}(x_i)\bigr)<\eta.$
Given any \(0\le t<mp_i\), write \(t=jm+s\) with \(0\le j<p_i\) and \(0\le s<m\).
Then, by the choice of \(\eta\),
\[
d\bigl(f_1^t(y),f_1^t(x_i)\bigr)
=
d\bigl(f_1^s(f_1^{jm}(y)),f_1^s(f_1^{jm}(x_i))\bigr)
<\delta.
\]
Hence \(y\in B_{mp_i,\boldsymbol f}(x_i,\delta)\).
Therefore,
$\{B_{mp_i,\boldsymbol f}(x_i,\delta)\}_i$
is a cover of \(Z\) with respect to \(\boldsymbol f\), and since
\(p_i\in[N,N/\theta]\), we have
$mp_i\in[mN,mN/\theta].$
Next, by \eqref{pr-eq},
\begin{equation}\label{eq:step5}
\sum_i
\exp\Bigl(
-\alpha p_i + S_{p_i}^{\boldsymbol f^m}(S_m\varphi)(x_i)
\Bigr)
=
\sum_i
\exp\Bigl(
-\frac{\alpha}{m}\,mp_i + S_{mp_i}^{\boldsymbol f}\varphi(x_i)
\Bigr).
\end{equation}

For $r\in\{0,\dots,m-1\}$, set $K=mN+r$ and
$\Omega(\delta):=\sum_{\ell=0}^{r-1}\omega(2^\ell\delta).$
We extend the length of each Bowen ball 
$B_{mp_i,\boldsymbol{f}}(x_i,\delta)$
from $mp_i$ to $mp_i+r$ without losing coverage.
Since $X$ is compact, there exist points $z_1,\ldots,z_L\in X$ such that
\[
X \subseteq \bigcup_{j=1}^{L} B_d\bigl(z_j,\delta\bigr).
\]
For each \(i\) and each \(j\in\{1,\ldots,L\}\) such that
\[
B_{mp_i,\boldsymbol{f}}(x_i,\delta)\cap
f_1^{-mp_i}B_d(z_j,\delta)\neq\emptyset,
\]
choose a point \(y_{i+1,j}\) from this nonempty intersection.
It is easy to verify that
\[
B_{mp_i,\boldsymbol{f}}(x_i,\delta)
\subseteq 
\bigcup_{j}
B_{mp_i+1,\boldsymbol{f}}\bigl(y_{i+1,j},2\delta\bigr).
\]
Next, for any $0 \le k \le mp_i-1$ and any such $j$, we have
\[
\bigl|\varphi\bigl(f_1^{k}(y_{i+1,j})\bigr) - \varphi\bigl(f_1^{k}(x_i)\bigr)\bigr|
\leq \omega\bigl(\delta\bigr),
\]
which implies
\[
S_{mp_i}^{\boldsymbol{f}}\varphi(x_i) +mp_i\,\omega\bigl(\delta\bigr)
\;\geq\; S_{mp_i}^{\boldsymbol{f}}\varphi(y_{i+1,j}).
\]
Hence
\[
S_{mp_i}^{\boldsymbol{f}}\varphi(x_i) + mp_i\,\omega\bigl(\delta\bigr) + M
\;\geq\; S_{mp_i+1}^{\boldsymbol{f}}\varphi(y_{i+1,j}),
\]
so that
\[
e^{-\alpha}\,e^{mp_i\omega(\delta) + M}\,
e^{-\alpha mp_i+S_{mp_i}^{\boldsymbol{f}}\varphi(x_i)}
\;\geq\;
e^{-\alpha(mp_i+1)+S_{mp_i+1}^{\boldsymbol{f}}\varphi(y_{i+1,j})}.
\]
Summing over all admissible indices $i,j$, we obtain
\[
L \sum_{i} e^{-\alpha}\,e^{mp_i\omega(\delta)+M}
e^{-\alpha mp_i+S_{mp_i}^{\boldsymbol{f}}\varphi(x_i)}
\;\geq\;
\sum_{i}\sum_{j} e^{-\alpha(mp_i+1)+S_{mp_i+1}^{\boldsymbol{f}}\varphi(y_{i+1,j})}.
\]
By repeating this construction \(r\) times, we obtain for each \(i\) a finite family
\[
\mathcal F_i=
\left\{
B_{mp_i+r,\boldsymbol f}(y_{i+r,j},2^r\delta):
j=1,\ldots,M_i,\ M_i\le L^r
\right\},
\]
such that
\[
B_{mp_i,\boldsymbol f}(x_i,\delta)
\subseteq
\bigcup_{j=1}^{M_i}
B_{mp_i+r,\boldsymbol f}(y_{i+r,j},2^r\delta).
\]
For each \(0\le \ell\le r-1\), we have
\[
\left|
S_{mp_i+\ell}^{\boldsymbol f}\varphi(y_{i+\ell+1,j})
-
S_{mp_i+\ell}^{\boldsymbol f}\varphi(x_i)
\right|
\le
(mp_i+\ell)\,\omega(2^\ell\delta).
\]
Summing over \(\ell=0,\ldots,r-1\), we obtain
\[
S_{mp_i+r}^{\boldsymbol f}\varphi(y_{i+r,j})
\le
S_{mp_i}^{\boldsymbol f}\varphi(x_i)
+
\sum_{\ell=0}^{r-1}(mp_i+\ell)\omega(2^\ell\delta)
+
rM.
\]
Since
\[
\sum_{\ell=0}^{r-1}(mp_i+\ell)\omega(2^\ell\delta)
\le
(mp_i+r)\Omega(\delta),
\]
it follows that
\[
S_{mp_i+r}^{\boldsymbol f}\varphi(y_{i+r,j})
\le
S_{mp_i}^{\boldsymbol f}\varphi(x_i)
+
(mp_i+r)\Omega(\delta)
+
rM.
\]
Therefore,
\[
e^{-(\alpha+\Omega(\delta))(mp_i+r)
+S_{mp_i+r}^{\boldsymbol f}\varphi(y_{i+r,j})}
\le
e^{(M-\alpha)r}
e^{-\alpha mp_i+S_{mp_i}^{\boldsymbol f}\varphi(x_i)}.
\]
Summing over \(i,j\) and using \(M_i\le L^r\), we obtain
\begin{equation}\label{eq:step6}
\sum_i
e^{-\alpha mp_i+S_{mp_i}^{\boldsymbol f}\varphi(x_i)}
\ge
C
\sum_i\sum_{j=1}^{M_i}
e^{-(\alpha+\Omega(\delta))(mp_i+r)
+S_{mp_i+r}^{\boldsymbol f}\varphi(y_{i+r,j})},
\end{equation}
where
$C=L^{-r}e^{(\alpha-M)r}.$
As $\mathcal F=\bigcup_i\mathcal F_i$ covers $Z$ and its elements have lengths in $[K,\,K/\theta]$,
combining \eqref{eq:step5} and \eqref{eq:step6} and taking the infimum over all initial covers $\mathcal G_{N,\boldsymbol{f}^m}$, we obtain
\[
\mathcal{M}_{\boldsymbol{f}^m}(Z,\alpha m,S_m\varphi,\eta,N,\theta)
\;\ge\;
C\,
\mathcal{M}_{\boldsymbol{f}}\!\left(
Z,\alpha+\Omega(\delta),\varphi,\,
2^r\delta,\,
K,\theta\right).
\]
Taking \(\liminf_{N\to\infty}\) and \(\limsup_{N\to\infty}\), respectively, we obtain
\[
\underline{\mathfrak{m}}_{\boldsymbol{f}^m}(Z,\alpha m,S_m\varphi,\eta,\theta)
\;\ge\;
C\,
\underline{\mathfrak{m}}_{\boldsymbol{f}}\!\left(
Z,\alpha+\Omega(\delta),\varphi,\,
2^r\delta,\theta\right),
\]
and
\[
\overline{\mathfrak{m}}_{\boldsymbol{f}^m}(Z,\alpha m,S_m\varphi,\eta,\theta)
\;\ge\;
C\,
\overline{\mathfrak{m}}_{\boldsymbol{f}}\!\left(
Z,\alpha+\Omega(\delta),\varphi,\,
2^r\delta,\theta\right).
\]
This implies that
\[
\underline{P}'(\boldsymbol{f}^m,Z,S_m\varphi,\eta,\theta) \geq m\underline{P}'(\boldsymbol{f},Z,\varphi,2^r\delta,\theta)-m\Omega(\delta),
\]
\[
\overline{P}'(\boldsymbol{f}^m,Z,S_m\varphi,\eta,\theta) \geq m\overline{P}'(\boldsymbol{f},Z,\varphi,2^r\delta,\theta)-m\Omega(\delta).
\]
Finally, as $\eta \to 0$ we have 
$\delta\to0$ and $\Omega(\delta)\to0$, 
and therefore the desired equalities for 
$\underline P$ and $\overline P$ follow.

Combining Parts~I and~II completes the proof.
\end{proof}

\begin{proposition}\label{le5.1-theta-refined}
Let \( (X,d) \) be a compact metric space and \( \boldsymbol{f} \) be a sequence of continuous self-maps of \( X \).
For any $Z\subseteq X, k\in\mathbb N, \varphi \in C(X,\mathbb{R})$ and $\theta\in [0,1]$, we have
\[
\underline{P}\bigl(\boldsymbol{f}_{k},Z,\varphi,\theta\bigr)
=\underline{P}\bigl(\boldsymbol{f}_{k+1},f_k(Z),\varphi,\theta\bigr),
\quad
\overline{P}\bigl(\boldsymbol{f}_{k},Z,\varphi,\theta\bigr)
=\overline{P}\bigl(\boldsymbol{f}_{k+1},f_k(Z),\varphi,\theta\bigr).
\]
\end{proposition}

\begin{proof}
We only give the proof for the lower $\theta$-intermediate pressure. 
The argument for the upper $\theta$-intermediate pressure is entirely analogous. Set \(M=\|\varphi\|\). We divide the proof into two parts as follows.

\medskip
\noindent\textbf{Part I.}
We prove that
\[
\underline{P}\bigl(\boldsymbol{f}_{k},Z,\varphi,\theta\bigr)
\geq 
\underline{P}\bigl(\boldsymbol{f}_{k+1},f_k(Z),\varphi,\theta\bigr).
\]
\noindent\emph{Case 1: $\theta\in(0,1]$.}
Fix \(\alpha \in \mathbb{R}\) and \(N \ge 2\). Choose a family of Bowen balls
\(\mathcal F=\bigl\{B_{n_i,\boldsymbol{f}_{k}}(x_i,\delta)\bigr\}_i\)
with respect to \(\boldsymbol{f}_{k}\) such that
\[
Z\subseteq\bigcup_i B_{n_i,\boldsymbol{f}_{k}}(x_i,\delta),
\qquad 
N\le n_i \le \frac{N}{\theta} \quad\text{for all }i.
\]
Let $R=(N-1)/\theta$ and set $q_i=\min\{n_i,\lfloor R\rfloor+1\}$, so that
\(q_i \in [N,\lfloor R\rfloor+1]\) and 
$Z\subseteq\bigcup_i B_{q_i,\boldsymbol{f}_{k}}(x_i,\delta).$
A direct computation shows that
\[
B_{q_i,\boldsymbol{f}_{k}}(x_i,\delta) \subseteq  f_{k}^{-1}
\left(\bigcap_{t=0}^{q_i-2} f_{k+1}^{-t}(B_{d}(f_{k+1}^t(f_{k}(x_i)),\delta))\right),
\]
and therefore
\[
f_{k}(Z) \subseteq \bigcup_i B_{q_i-1,\boldsymbol{f}_{k+1}}(f_{k}(x_i),\delta),\qquad
q_i-1\in[N-1,R].
\]
Since 
\(
S_{q_i-1}^{\boldsymbol{f}_{k+1}}\varphi(f_{k}(x_i))=S_{q_i}^{\boldsymbol{f}_{k}}\varphi(x_i)-\varphi(x_i),
\)
we have
\[
S_{n_i}^{\boldsymbol{f}_{k}}\varphi(x_i) \geq S_{q_i-1}^{\boldsymbol{f}_{k+1}}\varphi(f_{k}(x_i))-(n_i-q_i+1)M.
\]
Furthermore, \(n_i-q_i<1/\theta\), and hence
\[
\begin{aligned}
e^{-\alpha n_i+S_{n_i}^{\boldsymbol{f}_{k}}\varphi(x_i)}
&\; \ge \;
e^{-\alpha q_i}e^{-|\alpha|/\theta
   +S_{q_i-1}^{\boldsymbol{f}_{k+1}}\varphi(f_k(x_i))
   -(1/\theta+1)M}
\\[0.4em]
&\;\ge\;
e^{-\alpha (q_i-1)+S_{q_i-1}^{\boldsymbol{f}_{k+1}}\varphi(f_k(x_i))}
\;e^{-(|\alpha|+M)(1/\theta+1)}.
\end{aligned}
\]
It follows that
\begin{align*}
e^{(|\alpha|+M)(\frac1\theta+1)}
\sum_{i}
   e^{-\alpha n_i+S_{n_i}^{\boldsymbol{f}_{k}}\varphi(x_i)}
&\;\ge\;
\sum_{i}
   e^{-\alpha (q_i-1)+S_{q_i-1}^{\boldsymbol{f}_{k+1}}\varphi(f_{k}(x_i))}
\\[0.4em]
&\;\ge\;
\mathcal{M}_{\boldsymbol{f}_{k+1}}\bigl(f_k(Z),\alpha,\varphi,\delta,N-1,\theta\bigr).
\end{align*}
Taking the infimum over $\mathcal F$ and then $\liminf_{N\to\infty}$ gives
\[
e^{(|\alpha|+M)(\frac1\theta+1)}\,
\underline{\mathfrak{m}}_{\boldsymbol{f}_{k}}(Z,\alpha,\varphi,\delta,\theta)
\;\ge\;
\underline{\mathfrak{m}}_{\boldsymbol{f}_{k+1}}(f_k(Z),\alpha,\varphi,\delta,\theta).
\]
Thus
\[
\underline{P}'\bigl(\boldsymbol{f}_{k},Z,\varphi,\delta,\theta\bigr)
\geq 
\underline{P}'\bigl(\boldsymbol{f}_{k+1},f_k(Z),\varphi,\delta,\theta\bigr).
\]
\smallskip
\noindent\emph{Case 2: $\theta=0$.}
In this case, the admissible lengths in the definition of 
$\mathcal M_{\boldsymbol{f}_{k}}(\cdot,\alpha,\varphi,\delta,N,0)$ satisfy $n_i\ge N$ without any upper bound.
Fix $\alpha\in \mathbb{R}$ and $N \ge 2$, and let
$\mathcal F=\{B_{n_i,\boldsymbol{f}_{k}}(x_i,\delta)\}_i$
be a family of Bowen balls with respect to $\boldsymbol{f}_{k}$ such that $n_i\ge N$ for all $i$ and
$Z\subseteq\bigcup_i B_{n_i,\boldsymbol{f}_{k}}(x_i,\delta).$
As in Case~1, for each $i$ we have
\[
B_{n_i,\boldsymbol{f}_{k}}(x_i,\delta)
\subseteq 
f_k^{-1}\Bigl(
   \bigcap_{t=0}^{n_i-2} f_{k+1}^{-t}
   \bigl(B_d(f_{k+1}^t(f_k(x_i)),\delta)\bigr)
 \Bigr)
=
f_k^{-1}\bigl(B_{n_i-1,\boldsymbol{f}_{k+1}}(f_k(x_i),\delta)\bigr),
\]
and hence
\[
f_k(Z)\subseteq\bigcup_i B_{n_i-1,\boldsymbol{f}_{k+1}}(f_k(x_i),\delta),
\qquad n_i-1\ge N-1.
\]
Then
\[
-\alpha(n_i-1)+S_{n_i-1}^{\boldsymbol{f}_{k+1}}\varphi(f_k(x_i))
\le
-\alpha n_i+S_{n_i}^{\boldsymbol{f}_{k}}\varphi(x_i)+\alpha+M,
\]
which implies
\[
e^{-\alpha(n_i-1)+S_{n_i-1}^{\boldsymbol{f}_{k+1}}\varphi(f_k(x_i))}
\le
e^{\alpha+M}\,e^{-\alpha n_i+S_{n_i}^{\boldsymbol{f}_{k}}\varphi(x_i)}.
\]
Thus,
\[
\mathcal M_{\boldsymbol{f}_{k+1}}(f_k(Z),\alpha,\varphi,\delta,N-1,0)
\le
e^{\alpha+M}
\mathcal M_{\boldsymbol{f}_{k}}(Z,\alpha,\varphi,\delta,N,0).
\]
Taking the infimum over $\mathcal F$ and then 
$\liminf_{N\to\infty}$ gives
\[
\underline{\mathfrak m}_{\boldsymbol{f}_{k}}(Z,\alpha,\varphi,\delta,0)
\ge
e^{-(\alpha+M)}\,
\underline{\mathfrak m}_{\boldsymbol{f}_{k+1}}(f_k(Z),\alpha,\varphi,\delta,0).
\]
Consequently,
\[
\underline{P}'(\boldsymbol{f}_{k},Z,\varphi,\delta,0)
\ge
\underline{P}'(\boldsymbol{f}_{k+1},f_k(Z),\varphi,\delta,0).
\]
Combining the two cases and letting $\delta\to0$ yields
\[
\underline{P}(\boldsymbol{f}_{k},Z,\varphi,\theta)
\ge
\underline{P}(\boldsymbol{f}_{k+1},f_k(Z),\varphi,\theta).
\]
This completes Part~I.

\medskip
\noindent\textbf{Part~II.}
We prove the reverse inequality
\[
\underline{P}\bigl(\boldsymbol{f}_{k},Z,\varphi,\theta\bigr)
\leq 
\underline{P}\bigl(\boldsymbol{f}_{k+1},f_k(Z),\varphi,\theta\bigr).
\]
Since $X$ is compact, there exist points $z_1,\ldots,z_L\in X$ such that
\[
X \subseteq \bigcup_{j=1}^{L} B_d\bigl(z_j,\delta\bigr).
\]
We first consider the case $\theta \in (0,1]$.
Fix \(\alpha \in \mathbb{R}\). Choose a family of Bowen balls
\(\mathcal F=\bigl\{B_{n_i,\boldsymbol{f}_{k+1}}(x_i,\delta)\bigr\}_i\)
with respect to \(\boldsymbol{f}_{k+1}\) such that
\[
f_k(Z)\subseteq\bigcup_i B_{n_i,\boldsymbol{f}_{k+1}}(x_i,\delta),
\qquad 
N\le n_i \le \frac{N}{\theta} \quad\text{for all }i.
\]
It follows that
\[
Z 
\subseteq f_k^{-1}\bigl(f_k(Z)\bigr) 
\subseteq
\bigcup_{j=1}^{L}\;\bigcup_{i}
\Bigl(
B_d(z_j,\delta)
\;\cap\;
\bigcap_{t=1}^{n_i}
   f_k^{-t}\bigl( B_d(f_{k+1}^{t-1}(x_i),\delta) \bigr)
\Bigr).
\]
For each \(i\) and each \(j\in\{1,\ldots,L\}\) for which
\[
B_d(z_j,\delta) \cap 
\bigcap_{t=1}^{n_i} f_k^{-t}
\bigl(B_d(f_{k+1}^{t-1}(x_i),\delta)\bigr)\neq\emptyset,
\]
choose a point \(y_{i,j}\) from this nonempty intersection.
For such $(i,j)$ we have
\[
B_d\bigl(z_j,\delta\bigr) \cap 
\bigcap_{t=1}^{n_i} f_k^{-t}\bigl(B_{d}(f_{k+1}^{t-1}(x_i),\delta)\bigr)
\subseteq B_{n_i+1,\boldsymbol{f}_k}(y_{i,j},2\delta).
\]
Hence the family $\{B_{n_i+1,\boldsymbol{f}_k}(y_{i,j},2\delta)\}_{(i,j)}$,
indexed over all pairs $(i,j)$ with non-empty intersection above, 
covers $Z$ with respect to $\boldsymbol{f}_k$. 
Moreover, for each such \((i,j)\) and every \(1\le t\le n_i\),
\[
d\bigl(f_k^t(y_{i,j}),f_{k+1}^{t-1}(x_i)\bigr)<\delta.
\]
Thus
\[
S_{n_i+1}^{\boldsymbol f_k}\varphi(y_{i,j})
\le
S_{n_i}^{\boldsymbol f_{k+1}}\varphi(x_i)
+n_i\omega(\delta)+M.
\]
Consequently,
\[
e^{-\alpha(n_i+1)+S_{n_i+1}^{\boldsymbol f_k}\varphi(y_{i,j})}
\le
e^{-\alpha+M}
e^{-(\alpha-\omega(\delta))n_i+
S_{n_i}^{\boldsymbol f_{k+1}}\varphi(x_i)}.
\]
Then
\begin{align*}
L\, e^{-\alpha+M}
   \sum_{i}
      e^{- (\alpha-\omega(\delta))n_i+S_{n_i}^{\boldsymbol{f}_{k+1}}\varphi(x_i)}
& \geq \sum_{j}
   \sum_{i}
      e^{-\alpha(n_i+1)+S_{n_i+1}^{\boldsymbol{f}_{k}}\varphi(y_{i,j})}\\[0.3em]
&\ge \mathcal{M}_{\boldsymbol{f}_{k}}(Z,\alpha,\varphi,2\delta,N+1,\theta).
\end{align*}
Taking the infimum over $\mathcal F$ and then $\liminf_{N\to\infty}$ gives
\[
\underline{\mathfrak{m}}_{\boldsymbol{f}_{k}}(Z,\alpha,\varphi,2\delta,\theta)\le 
L\, e^{-\alpha+M}\,
\underline{\mathfrak{m}}_{\boldsymbol{f}_{k+1}}(f_{k}(Z),\alpha-\omega(\delta),\varphi,\delta,\theta),
\]
which implies
\[
\underline{P}'\bigl(\boldsymbol{f}_{k},Z,\varphi,2\delta,\theta\bigr)
\leq 
\underline{P}'\bigl(\boldsymbol{f}_{k+1},f_k(Z),\varphi,\delta,\theta\bigr)+\omega(\delta).
\]
Letting $\delta\to0$ yields
\[
\underline{P}\bigl(\boldsymbol{f}_{k},Z,\varphi,\theta\bigr)
\leq 
\underline{P}\bigl(\boldsymbol{f}_{k+1},f_k(Z),\varphi,\theta\bigr).
\] 
The case $\theta=0$ follows from the same argument, noting that no upper bound on $n_i$ is required.

Combining Parts~I and~II completes the proof.
\end{proof}

\begin{corollary}\label{momo-theta-pressure}
For the notions of forward invariant, backward invariant and invariant subsets
of nonautonomous systems, we follow \cite[Corollary 3.6]{ju2026entropy}.
For $Z\subseteq X$, assume one of the following:
\begin{itemize}
\item $Z$ is $\boldsymbol f$-forward invariant: $f_k(Z)\subseteq Z$ for all $k \in \mathbb{N}$;
\item $Z$ is $\boldsymbol f$-backward invariant: $Z\subseteq f_k(Z)$ for all $k\in \mathbb{N}$;
\item $Z$ is $\boldsymbol f$-invariant: $f_k(Z)=Z$ for all $k\in \mathbb{N}$.
\end{itemize}
Then for any $\varphi\in C(X,\mathbb R)$, \(1 \le i < j < \infty\) and $\theta \in [0,1]$, the following hold:
\begin{enumerate}
\item[(1)] If $Z$ is forward invariant, then
\[
\underline P(\boldsymbol f_i,Z,\varphi,\theta)
\le 
\underline P(\boldsymbol f_j,Z,\varphi,\theta),
\qquad
\overline P(\boldsymbol f_i,Z,\varphi,\theta)
\le 
\overline P(\boldsymbol f_j,Z,\varphi,\theta).
\]

\item[(2)] If $Z$ is backward invariant, then
\[
\underline P(\boldsymbol f_i,Z,\varphi,\theta)
\ge 
\underline P(\boldsymbol f_j,Z,\varphi,\theta),
\qquad
\overline P(\boldsymbol f_i,Z,\varphi,\theta)
\ge 
\overline P(\boldsymbol f_j,Z,\varphi,\theta).
\]

\item[(3)] If $Z$ is invariant, then
\[
\underline P(\boldsymbol f_i,Z,\varphi,\theta)
=
\underline P(\boldsymbol f_j,Z,\varphi,\theta),
\qquad
\overline P(\boldsymbol f_i,Z,\varphi,\theta)
=
\overline P(\boldsymbol f_j,Z,\varphi,\theta).
\]
\end{enumerate}
\end{corollary}

\begin{proof}
By Proposition~\ref{le5.1-theta-refined}, for every $k\in\mathbb N$,
\[
\underline P(\boldsymbol f_k,Z,\varphi,\theta)
=
\underline P(\boldsymbol f_{k+1},f_k(Z),\varphi,\theta),
\]
and similarly for $\overline P$.
Iterating the above identity from $k=i$ to $j-1$, we obtain
\[
\underline P(\boldsymbol f_i,Z,\varphi,\theta)
=
\underline P(\boldsymbol f_j,f_i^{\,j-i}(Z),\varphi,\theta),
\]
and analogously for $\overline P$.
Forward (resp.\ backward) invariance gives
$f_i^{\,j-i}(Z)\subseteq Z$ (resp.\ $Z\subseteq f_i^{\,j-i}(Z)$),
implying the desired inequalities.  
Equalities hold when $Z$ is invariant.
\end{proof}

\begin{proposition}\label{commute-theta-pressure}
Let $f_1,f_2$ be continuous maps on a compact space $X$, and let $Z\subseteq X$.
Assume $Z$ is $\{f_1,f_2\}$-forward invariant or $\{f_1,f_2\}$-backward invariant.
Then for all $\varphi\in C(X,\mathbb R)$ and $\theta\in[0,1]$,
\[
\underline P(f_1\circ f_2,Z,\varphi+\varphi \circ f_2,\theta)
=\underline P(f_2 \circ f_1,Z,\varphi+\varphi \circ f_1,\theta),
\]
\[
\overline P(f_1 \circ f_2,Z,\varphi+\varphi \circ f_2,\theta)
=\overline P(f_2 \circ f_1,Z,\varphi+\varphi \circ f_1,\theta).
\]
\end{proposition}
\begin{proof}
Let
\[
\boldsymbol f=\{f_1,f_2,f_1,f_2,\ldots\}, \qquad
\boldsymbol g=\{f_2,f_1,f_2,f_1,\ldots\}.
\]
By Corollary~\ref{momo-theta-pressure},
\[
\underline P(\boldsymbol f,Z,\varphi,\theta)
=
\underline P(\boldsymbol g,Z,\varphi,\theta),
\]
and the same for $\overline P$.
Applying Proposition~\ref{pr-theta} with $m=2$ yields
\[
\underline P(f_1 \circ f_2,Z,\varphi+\varphi \circ f_2,\theta)
=
\underline P(f_2 \circ f_1,Z,\varphi+\varphi \circ f_1,\theta),
\]
and similarly for $\overline P$.
\end{proof}

\section{Topological conjugacy}\label{sec:factor}

Let $(X,\boldsymbol{f})$ and $(Y,\boldsymbol{g})$ be two NDSs, where $\boldsymbol{f}=\{f_i:X\to X\}_{i=1}^{\infty}$ and $\boldsymbol{g}=\{g_i:Y\to Y\}_{i=1}^{\infty}$ are sequences of continuous maps. 
A sequence of continuous surjective maps $\boldsymbol{\pi}=\{\pi_i:X\to Y\}_{i=1}^{\infty}$ is called a \emph{semiconjugacy} from $(X,\boldsymbol{f})$ to $(Y,\boldsymbol{g})$ if \[ \pi_{i+1}\circ f_i = g_i\circ \pi_i \qquad\text{for all } i\ge1. \] In this case, $(Y,\boldsymbol{g})$ is called a factor of $(X,\boldsymbol{f})$. Moreover, if each $\pi_i$ is a homeomorphism, then $\boldsymbol{\pi}$ is called a \emph{conjugacy} between the two systems; in this
situation, the inverse mappings $\pi_i^{-1}$ form a sequence
$
\boldsymbol{\pi}^{-1}=\{\pi_i^{-1}:Y\to X\}_{i=1}^{\infty},
$
which provides a semiconjugacy from $(Y,\boldsymbol{g})$ back to
$(X,\boldsymbol{f})$, and hence the systems are (topologically) conjugate.

In this section, we restrict ourselves to \emph{time-independent semiconjugacies}, 
that is, $\pi_i=\pi$ for all $i\ge1$. Then the semiconjugacy condition reduces to
$\pi\circ f_i = g_i\circ \pi$ for all $i\ge1.$
Then $\pi:X\to Y$ is called a factor map from 
$(X,\boldsymbol{f})$ to $(Y,\boldsymbol{g})$.

\begin{theorem}\label{con}
Let $(X,d)$ and $(Y,\rho)$ be compact metric spaces, and let 
$\boldsymbol{f}$ and $\boldsymbol{g}$ be sequences of continuous self-maps on $X$ and $Y$, respectively.  
Suppose $\pi$ is a factor map from $(X,\boldsymbol{f})$ to $(Y,\boldsymbol{g})$.  
Then for every subset $Z\subseteq X$, $\theta\in[0,1]$ and every $\varphi\in C(Y,\mathbb{R})$,
\[
\underline{P}(\boldsymbol{f},Z,\varphi \circ \pi,\theta)\;\ge\;\underline{P}(\boldsymbol{g},\pi(Z),\varphi, \theta),
\qquad
\overline{P}(\boldsymbol{f},Z,\varphi \circ \pi,\theta)\;\ge\;\overline{P}(\boldsymbol{g},\pi(Z),\varphi, \theta).
\]
If $\pi$ is a conjugacy, then equality holds:
\[
\underline{P}(\boldsymbol{f},Z,\varphi \circ \pi,\theta)\;=\;\underline{P}(\boldsymbol{g},\pi(Z),\varphi,\theta),
\qquad
\overline{P}(\boldsymbol{f},Z,\varphi \circ \pi,\theta)\;=\;\overline{P}(\boldsymbol{g},\pi(Z),\varphi,\theta).
\]
\end{theorem}

\begin{proof}
Since $\pi$ is continuous and $X$ is compact, $\pi$ is uniformly continuous.
Hence for any $\varepsilon>0$ there exists $\delta=\delta(\varepsilon)\in(0,\varepsilon]$ such that
\[
d(x,y)<\delta \quad\Longrightarrow\quad 
\rho(\pi(x),\pi(y))<\varepsilon.
\]
We first consider the case \(\theta\in(0,1]\).
Let 
$\mathcal G_{N,\boldsymbol{f}}
=\bigl\{\,B_{n_i,\boldsymbol{f}}(x_i,\delta)\,\bigr\}_i$
be a cover of \(Z\), where each \(B_{n_i,\boldsymbol{f}}(x_i,\delta)\) is the \((n_i,\delta)\)-Bowen ball 
with respect to \(\boldsymbol{f}\) and \(n_i\in[N,\,N/\theta]\).
Then for every 
\(B_{n_i,\boldsymbol{f}}(x_i,\delta)\in\mathcal G_{N,\boldsymbol{f}}\), we have
\[
B_{n_i,\boldsymbol{f}}(x_i,\delta)
=\bigcap_{p=0}^{n_i-1} f_1^{-p}\left(B_d(f_1^{p}(x_i),\delta)\right).
\]
For any \(y\in f_1^{-p}\left(B_d(f_1^{p}(x_i),\delta)\right)\) with \(0\le p\le n_i-1\),
we have \(d(f_1^{p}(x_i),f_1^{p}(y))<\delta\), and thus
\[
\rho\!\left(\pi(f_1^{p}(x_i)),\,\pi(f_1^{p}(y))\right)<\varepsilon.
\]
Therefore
\[
y\in(\pi\circ f_1^{p})^{-1}
   \left(B_{\rho}(\pi (f_1^{p}(x_i)),\varepsilon)\right)
   =(g_1^{p}\pi)^{-1}\left(B_{\rho}(g_1^{p}\pi(x_i),\varepsilon)\right).
\]
Consequently,
\[
B_{n_i,\boldsymbol{f}}(x_i,\delta)
\subseteq
\pi^{-1}\!\left(B_{n_i,\boldsymbol{g}}(\pi(x_i),\varepsilon)\right).
\]
Hence the family
\[
\mathcal H_{N,\boldsymbol{g}}
=\left\{
B_{n_i,\boldsymbol{g}}(\pi(x_i),\varepsilon):
B_{n_i,\boldsymbol{f}}(x_i,\delta)\in\mathcal G_{N,\boldsymbol{f}}
\right\}
\]
covers \(\pi(Z)\). 
Since $\pi(f_1^p(x_i))=g_1^p(\pi(x_i))$, we have for all $n_i$,
\[
S_{n_i}^{\boldsymbol{f}}(\varphi\circ\pi)(x_i)
=
\sum_{p=0}^{n_i-1}(\varphi\circ\pi)(f_1^p(x_i))
=
\sum_{p=0}^{n_i-1}\varphi(g_1^p(\pi(x_i)))
=
S_{n_i}^{\boldsymbol{g}}\varphi(\pi(x_i)).
\]
Furthermore, we obtain
\begin{align*}
\mathcal M_{\boldsymbol{f}}(Z,\alpha,\varphi\circ\pi,\delta,N,\theta)
&=\inf_{\mathcal{G}_{N,\boldsymbol{f}}}
  \sum_{B_{n_i,\boldsymbol{f}}(x_i,\delta)\in\mathcal{G}_{N,\boldsymbol{f}}}
  e^{-\alpha n_i + S_{n_i}^{\boldsymbol{f}}(\varphi\circ\pi)(x_i)} \\[0.6em]
&=\inf_{\mathcal{G}_{N,\boldsymbol{f}}}
  \sum_{B_{n_i,\boldsymbol{f}}(x_i,\delta)\in\mathcal{G}_{N,\boldsymbol{f}}}
  e^{-\alpha n_i + S_{n_i}^{\boldsymbol{g}}\varphi(\pi(x_i))} \\[0.6em]
&\ge
\inf_{\mathcal{H}_{N,\boldsymbol{g}}}
  \sum_{B_{n_i,\boldsymbol{g}}(\pi(x_i),\varepsilon)\in\mathcal{H}_{N,\boldsymbol{g}}}
  e^{-\alpha n_i + S_{n_i}^{\boldsymbol{g}}\varphi(\pi(x_i))} \\[0.6em]
&\ge
\mathcal M_{\boldsymbol{g}}(\pi(Z),\alpha,\varphi,\varepsilon,N,\theta).
\end{align*}
Taking $\liminf_{N\to\infty}$ in the above inequality yields
\[
\underline{\mathfrak{m}}_{\boldsymbol{f}}(Z,\alpha,\varphi\circ\pi,\delta,\theta)
\;\ge\;
\underline{\mathfrak{m}}_{\boldsymbol{g}}(\pi(Z),\alpha,\varphi,\varepsilon,\theta).
\]
Consequently, 
\[
\underline{P}'(\boldsymbol{f},Z,\varphi\circ\pi,\delta,\theta)
\;\ge\;
\underline{P}'(\boldsymbol{g},\pi(Z),\varphi,\varepsilon,\theta).
\]
Finally, letting $\varepsilon\to0$ gives
\[
\underline{P}(\boldsymbol{f},Z,\varphi\circ\pi,\theta)
\;\ge\;
\underline{P}(\boldsymbol{g},\pi(Z),\varphi,\theta).
\]
If $\pi$ is a homeomorphism,
the same argument applied to $\pi^{-1}$ yields the reverse inequality.
Therefore in this case we obtain equality:
\[
\underline{P}(\boldsymbol{f},Z,\varphi\circ\pi,\theta)
=
\underline{P}(\boldsymbol{g},\pi(Z),\varphi,\theta).
\]
The case \(\theta=0\) follows from the same argument, replacing the admissible condition
\(N\le n_i\le N/\theta\) by \(n_i\ge N\).

The proof of the assertion for the upper $\theta$-intermediate topological
pressure is entirely analogous.

\end{proof}

\begin{corollary}
If $g:X\to X$ is a homeomorphism commuting with $\boldsymbol{f}$ 
(i.e.\ $f_i\circ g=g\circ f_i$ for all $i\ge1$), then for any 
$Z\subseteq X$, $\varphi\in C(X,\mathbb{R})$, and $\theta\in[0,1]$, we have
\[
\underline{P}(\boldsymbol{f},Z,\varphi,\theta)
=
\underline{P}(\boldsymbol{f},g(Z),\varphi\circ g^{-1},\theta),
\qquad
\overline{P}(\boldsymbol{f},Z,\varphi,\theta)
=
\overline{P}(\boldsymbol{f},g(Z),\varphi\circ g^{-1},\theta).
\]
\end{corollary}

\begin{proof}
Since $f_i\circ g=g\circ f_i$ for all $i\ge1$, it follows that
\[
g\circ f_1^p = f_1^p\circ g, \qquad \forall\,p\ge0.
\]
Thus the constant sequence $\pi_i:=g$ defines a conjugacy of 
$(X,\boldsymbol{f})$ with itself, whose inverse is given by 
$\pi_i^{-1}=g^{-1}$.
Applying Theorem~\ref{con} completes the proof.
\end{proof}

Employing Bowen’s ideas and related developments
\cite{bowen1971entropy, kolyada1996topological, fang2012dimensions,
oprocha2011dimensional, li2013corrigendum, zhao2017formula}, we establish
an inequality for the $\theta$-intermediate topological pressures under a
factor map in the next theorem.  To this end, we recall the notion of
topological sup-entropy introduced by Kolyada and Snoha
\cite{kolyada1996topological}.

Let $(X,d)$ be a compact metric space, $Z\subseteq X$ a nonempty subset, and
$\boldsymbol{f}=\{f_i\}_{i\ge1}$ a sequence of equicontinuous self-maps of $X$.
For each $n\ge1$, define
\[
d_n^*(x,y)=\sup_i\,\max_{0\le j<n} d(f_i^j(x),f_i^j(y)), \qquad x,y\in X.
\]
Since $\boldsymbol{f}$ is equicontinuous, $d_n^*$ is equivalent to $d$, and thus $(X,d_n^*)$ is also compact.

A subset $E^*\subseteq X$ is called $(n,\varepsilon)^*$-separated if 
$d_n^*(x,y)>\varepsilon$ for any distinct $x,y\in E^*$.
A subset $F^*\subseteq X$ is called an $(n,\varepsilon)^*$-spanning set of $Z$ if for every $x\in Z$ there exists $y\in F^*$ such that $d_n^*(x,y)\le\varepsilon$.
Let \(s_n^*(\boldsymbol{f};Z;\varepsilon)\) denote the maximal cardinality of an $(n,\varepsilon)^*$-separated set in $Z$, and
\(r_n^*(\boldsymbol{f};Z;\varepsilon)\) the minimal cardinality of an $(n,\varepsilon)^*$-spanning set in $Z$.
The \emph{topological sup-entropy} of $\boldsymbol{f}$ on $Z$ is then defined by
\[
H(\boldsymbol{f};Z)
=\lim_{\varepsilon\to0}\limsup_{n\to\infty}\frac{1}{n}\log r_n^*(\boldsymbol{f};Z;\varepsilon)
=\lim_{\varepsilon\to0}\limsup_{n\to\infty}\frac{1}{n}\log s_n^*(\boldsymbol{f};Z;\varepsilon).
\]

\begin{theorem}
Let \((X, d)\) and \((Y, \rho)\) be compact metric spaces, \(\boldsymbol{f}\) be a sequence of equicontinuous maps from \(X\) into itself, \(\boldsymbol{g}\) be a sequence of equicontinuous maps from \(Y\) into itself. If \(\pi\) is a semiconjugacy from $(X,\boldsymbol{f})$ to $(Y,\boldsymbol{g})$, then for every nonempty $Z\subseteq X$, $\theta\in[0,1]$ and every $\varphi\in C(Y,\mathbb{R})$, we have,
\[
\underline{P}(\boldsymbol{f},Z,\varphi\circ\pi,\theta)
\le
\underline{P}(\boldsymbol{g},\pi(Z),\varphi,\theta)
+
\sup_{y \in Y} H\bigl(\boldsymbol{f}; \pi^{-1}(y)\bigr),
\]
\[
\overline{P}(\boldsymbol{f},Z,\varphi\circ\pi,\theta)
\le
\overline{P}(\boldsymbol{g},\pi(Z),\varphi,\theta)
+
\sup_{y \in Y} H\bigl(\boldsymbol{f}; \pi^{-1}(y)\bigr).
\]
\end{theorem}
\begin{proof}
Let \(a=\sup_{y \in Y} H\bigl(\boldsymbol{f}; \pi^{-1}(y)\bigr)\). If \(a= \infty\) there is nothing to prove. So we can assume that \(a< \infty\).
Fix \(\epsilon > 0\) and \(\tau > 0\). 
For the potential \(\varphi\in C(Y,\mathbb{R})\) and \(c>0\), we set
\[
\omega_\varphi(t)
:=\sup\{\lvert \varphi(y)-\varphi(y')\rvert : \rho(y,y')\le t\}.
\]
Since \(\varphi\circ\pi\) is a continuous function on \(X\), we also define
\[
\omega_{\varphi\circ\pi}(t)
:=\sup\{\lvert (\varphi\circ\pi)(x)-(\varphi\circ\pi)(x')\rvert : d(x,x')\le t\}.
\]
For each \( y \in Y \), choose \( m(y) \in \mathbb{N} \) such that
\[
a + \tau \geq  H\bigl(\boldsymbol{f}; \pi^{-1}(y); \epsilon\bigr) + \tau \geq \frac{1}{m(y)} \log r_{m(y)}^*\bigl(\boldsymbol{f}; \pi^{-1}(y); \epsilon\bigr),
\]
where 
\[
H\bigl(\boldsymbol{f}; \pi^{-1}(y); \epsilon\bigr)=\limsup_{n\to\infty}\frac{1}{n}\log r_n^*\bigl(\boldsymbol{f};\pi^{-1}(y);\epsilon\bigr).
\]
Let \( E_y^* \) be a \((m(y), \epsilon)^*\)-spanning set of \(\pi^{-1}(y)\) with respect to \(\boldsymbol{f}\), satisfying 
\[
\#\bigl(E_y^*\bigr) = r_{m(y)}^*\left(\boldsymbol{f}; \pi^{-1}(y); \epsilon\right).
\]
Define
\[
U_y= \left\{ u \in X : \exists z \in E_y^* \text{ such that } d_{m(y)}^{*}(u, z) < 2\epsilon \right\}.
\]
Then \( U_y \) is an open neighborhood of \(\pi^{-1}(y)\) and 
\[
\left(X \setminus  U_y\right) \cap \bigcap_{\gamma > 0} \pi^{-1}(\overline{B_\gamma(y)}) = \emptyset,
\]
where \( B_\gamma(y) = \{ y' \in Y : \rho(y', y) < \gamma \} \). By the finite intersection property of compact sets, there exists \( W_y = B_{\gamma(y)}(y) \) for which \( U_y \supseteq \pi^{-1}(W_y) \).

Since \( Y \) is compact, there exist \( y_1, \dots, y_p \) such that \( \{W_{y_1}, \dots, W_{y_p}\} \) cover \( Y \). Let \( \delta_1 > 0 \) be a Lebesgue number of the open cover \( \{W_{y_1}, \dots, W_{y_p}\} \) with respect to \(\rho\), and set \(0< \delta< \delta_1 / 2 \), \( M=\max_{1\le t\le p} m(y_t)\).

Now, for \( y \in Y \) and \( m \in \mathbb{N} \), by the claim in \cite[Theorem 4.6]{ju2026entropy}, there exist \( \ell(y) > 0 \) and \( v_1(y), \dots, v_{\ell(y)}(y) \in X \) such that

\[
\ell(y) \leq e^{(a + \tau)(m + M)} \quad \text{and} \quad \bigcup_{i=1}^{\ell(y)} B_{m,\boldsymbol f}(v_i(y), 4\epsilon) \supseteq \pi^{-1}(B_{m,\boldsymbol g}(y, \delta)),
\]
By discarding those indices \(i\) for which
\[
B_{m,\boldsymbol f}(v_i(y),4\epsilon)\cap 
\pi^{-1}\bigl(B_{m,\boldsymbol g}(y,\delta)\bigr)=\emptyset,
\]
we may assume that for all \(1\le i\le \ell(y)\),
\[
\pi\left(B_{m,\boldsymbol f}(v_i(y),4\epsilon)\right)
\cap B_{m,\boldsymbol g}(y,\delta)\neq \emptyset.
\]

For any \(n \in \mathbb{N}\) and sufficiently small \(\delta>0\), we let \(\left\{B_{n_j,\boldsymbol{g}}(w_j, \delta) \right\}_{j=1}^\infty\) be a cover of \( \pi(Z) \) such that for all $j$,
$n \le n_j \le n/\theta$ if $\theta\in(0,1]$,
and $n_j \ge n$ if $\theta=0$.
By the claim, for each \( B_{n_j,\boldsymbol{g}}(w_j, \delta) \), we have
\begin{itemize}
    \item  \( \ell(w_j) \leq e^{(a + \tau)(n_j + M)} \);
    \item \( \bigcup_{i=1}^{\ell(w_j)} B_{n_j,\boldsymbol{f}}\left(v_i(w_j), 4\epsilon \right) \supseteq \pi^{-1}\left(B_{n_j,\boldsymbol{g}}(w_j, \delta)\right) \);
    \item \( \pi\left(B_{n_j,\boldsymbol{f}}(v_i(w_j), 4\epsilon)\right) \cap B_{n_j,\boldsymbol{g}}(w_j, \delta) \neq \emptyset \) for any \( 1 \leq i \leq \ell(w_j) \).
\end{itemize}
This implies that
\[
\bigcup_{j=1}^\infty \bigcup_{i=1}^{\ell(w_j)} B_{n_j,\boldsymbol{f}}\left(v_i(w_j), 4\epsilon \right) \supseteq \bigcup_{j=1}^\infty \pi^{-1}\left(B_{n_j,\boldsymbol{g}}(w_j, \delta)\right) \supseteq \pi^{-1}(\pi(Z)) \supseteq Z.
\]
Then for any \(\alpha \in \mathbb{R}\), we have
\begin{align*}
\mathcal M_{\boldsymbol f}\left(Z,\alpha,\varphi\!\circ\!\pi,4\epsilon,n,\theta\right)
&\le 
\sum_{j=1}^{\infty}\sum_{i=1}^{\ell(w_j)}
   \exp\left(
      -\alpha n_j 
      + S_{n_j}^{\boldsymbol f}(\varphi \circ \pi)(v_i(w_j))
   \right) \\[0.6em]
&\le 
\sum_{j=1}^{\infty}\ell(w_j)\,
   \exp\left(
      -\alpha n_j
      + S_{n_j}^{\boldsymbol g}\varphi(w_j)
      + n_j\,\omega_{\delta,\epsilon}(\varphi)
   \right) \\[0.6em]
&\le 
e^{(a+\tau)M}
\sum_{j=1}^{\infty}
   \exp\left(
      -n_j\!\left(
         \alpha - a - \tau - \omega_{\delta,\epsilon}(\varphi)
      \right)
      + S_{n_j}^{\boldsymbol g}\varphi(w_j)
   \right),
\end{align*}
where
$\omega_{\delta,\varepsilon}(\varphi)=\omega_{\varphi}(\delta)
+\omega_{\varphi\circ\pi}(4\varepsilon).
$
Since the above inequality holds for any cover 
$\{B_{n_j,\boldsymbol g}(w_j,\delta)\}_{j=1}^\infty$ of $\pi(Z)$
satisfying the admissible condition in the definition of the 
$\theta$-intermediate pressure, we obtain
\[
\mathcal M_{\boldsymbol f}(Z,\alpha,\varphi\circ\pi,4\epsilon,n,\theta)
\;\le\;
e^{(a+\tau)M}\,
\mathcal M_{\boldsymbol g}\left(
\pi(Z),
\alpha-(a+\tau)-\omega_{\delta,\epsilon}(\varphi),
\varphi,\delta,n,\theta
\right).
\]
Taking the $\limsup$ as \(n\to\infty\) yields
\[
\overline{\mathfrak{m}}_{\boldsymbol f}(Z,\alpha,\varphi\circ\pi,4\epsilon,\theta)
\;\le\;
e^{(a+\tau)M}\,
\overline{\mathfrak{m}}_{\boldsymbol g}\left(
\pi(Z),
\alpha-(a+\tau)-\omega_{\delta,\epsilon}(\varphi),
\varphi,\delta,\theta
\right).
\]
Hence
\[
\overline P'(\boldsymbol f,Z,\varphi\circ\pi,4\epsilon,\theta)
\;\le\;
\overline P'(\boldsymbol g,\pi(Z),\varphi,\delta,\theta)
+a+\tau+\omega_{\delta,\epsilon}(\varphi).
\]
Letting first \(\epsilon\to0\) and \(\delta\to0\), and then letting \(\tau\to0\), we obtain
\[
\overline P(\boldsymbol f,Z,\varphi\circ\pi,\theta)
\;\le\;
\overline P(\boldsymbol g,\pi(Z),\varphi,\theta)+a.
\]
The argument for the lower \(\theta\)-intermediate pressure is analogous and therefore
\[
\underline P(\boldsymbol f,Z,\varphi\circ\pi,\theta)
\;\le\;
\underline P(\boldsymbol g,\pi(Z),\varphi,\theta)+a.
\]
\end{proof}

\section{Variational principles of $\theta$-intermediate topological pressures}\label{sec:variational}

In \cite{zhong2024variational}, Zhong and Chen established variational
principles for the Pesin-Pitskel pressure and for the lower and upper capacity pressures on compact subsets for autonomous systems. These three quantities may be viewed as the extremal cases $\theta=0$ and
$\theta=1$ of $\theta$-intermediate pressures. In this section, we extend their approach to NDSs and obtain a unified variational principle for all $\theta\in[0,1]$. To this end, we first introduce the $\theta$-intermediate measure-theoretic
pressures and then establish their relation with the corresponding
topological pressures.

Let \( \mathcal{M}(X) \) denote the set of all Borel probability measures on \(X\).
Given \( \mu\in\mathcal M(X) \), we call a family 
\(\mathcal{S}\subset \bigcup_{\varepsilon>0}\bigcup_{n\ge1}\{B_n(x,\varepsilon):x\in X\}\)
a \emph{\(\mu\)-cover} of \(X\) if \(\mu(\bigcup_{B\in\mathcal S}B)=1\).

Following the definitions of $\theta$-intermediate topological pressures, we define the corresponding $\theta$-intermediate measure-theoretic pressures. Given \( N \in \mathbb{N} \), \( \alpha \in \mathbb{R} \), \( \varepsilon > 0 \), and \( \varphi \in C(X, \mathbb{R}) \), define
\[
M_{\mu}(\boldsymbol{f}, \alpha, \varphi, \varepsilon, N, \theta ) = \inf \left\{\sum_i \exp\biggl(-\alpha\,n_i+\sup_{y \in B_{n_i}(x_i, \varepsilon)}S_{n_i}\varphi(y)\biggr) \right\},
\]
where the infimum is taken over all finite \( \mu \)-covers \(\mathcal{S}=\{B_{n_i}(x_i, \varepsilon)\}_{i}\) such
that \(N \le n_i \le N/\theta \) if $\theta>0$, and \(n_i \ge N\) if $\theta=0$.

Let
\[
\underline{m}_{\mu}(\boldsymbol{f}, \alpha, \varphi, \varepsilon,\theta) = \liminf_{N \to \infty}M_{\mu}(\boldsymbol{f}, \alpha, \varphi, \varepsilon, N, \theta ),
\]
\[
\overline{m}_{\mu}(\boldsymbol{f}, \alpha, \varphi, \varepsilon,\theta) = \limsup_{N \to \infty}M_{\mu}(\boldsymbol{f}, \alpha, \varphi, \varepsilon, N, \theta ).
\]

Define the numbers
\[
\underline{P}_{\mu}(\boldsymbol{f},  \varphi, \varepsilon,\theta) =\inf\{\alpha: \underline{m}_{\mu}(\boldsymbol{f}, \alpha, \varphi, \varepsilon,\theta)=0\}
=\sup\{\alpha: \underline{m}_{\mu}(\boldsymbol{f}, \alpha, \varphi, \varepsilon,\theta)=\infty\},
\]
\[
\overline{P}_{\mu}(\boldsymbol{f}, \varphi, \varepsilon,\theta) =\inf\{\alpha: \overline{m}_{\mu}(\boldsymbol{f}, \alpha, \varphi, \varepsilon,\theta)=0\}
=\sup\{\alpha: \overline{m}_{\mu}(\boldsymbol{f}, \alpha, \varphi, \varepsilon,\theta)=\infty\}.
\]

\begin{definition}
We call the following quantities
\[
\underline{P}_{\mu}(\boldsymbol{f},  \varphi, \theta)=\lim_{ \varepsilon \to 0}\underline{P}_{\mu}(\boldsymbol{f},  \varphi, \varepsilon,\theta),\qquad
\overline{P}_{\mu}(\boldsymbol{f},  \varphi, \theta)=\lim_{ \varepsilon \to 0}\overline{P}_{\mu}(\boldsymbol{f},  \varphi, \varepsilon,\theta)
\]
the lower and upper $\theta$-intermediate measure-theoretic pressures of \(\varphi\) with respect to \(\boldsymbol{f}\). 
\end{definition}

\begin{remark}
If we restrict the admissible covers in the definition of 
\(M(Z,\alpha,\varphi,\delta,N,\theta)\) to be finite ones, we obtain modified quantities
\(M^{\mathrm{fin}}\), \(\underline{\mathfrak m}^{\mathrm{fin}}\) and 
\(\overline{\mathfrak m}^{\mathrm{fin}}\), and the corresponding critical values
\(\underline{P}^{\mathrm{fin}}\) and \(\overline{P}^{\mathrm{fin}}\).
Moreover, if \(Z\subseteq X\) is compact, then
\[
\underline{P}^{\mathrm{fin}}(\boldsymbol f,Z,\varphi,\theta)
=
\underline P(\boldsymbol f,Z,\varphi,\theta),
\qquad
\overline{P}^{\mathrm{fin}}(\boldsymbol f,Z,\varphi,\theta)
=
\overline P(\boldsymbol f,Z,\varphi,\theta).
\]
\end{remark}

\begin{proposition}
Let \( (X,d) \) be a compact metric space and \( \boldsymbol{f} \) be a sequence of continuous self-maps of \( X \).
Then for any \(\mu \in \mathcal{M}(X)\), \(\varphi \in C(X,\mathbb{R})\) and \(\theta\in [0,1]\), we have 
\[
\begin{aligned}
\underline{P}_{\mu}(\boldsymbol{f},  \varphi, \theta)
&= \inf\bigl\{\underline{P}^{\mathrm{fin}}(\boldsymbol{f},Z,\varphi,\theta): \mu(Z) = 1,~Z\subseteq X\bigr\}, \\
\overline{P}_{\mu}(\boldsymbol{f},  \varphi, \theta)
&= \inf\bigl\{\overline{P}^{\mathrm{fin}}(\boldsymbol{f},Z,\varphi,\theta) : \mu(Z) = 1,~Z\subseteq X \bigr\}.
\end{aligned}
\]    
\end{proposition}
\begin{proof}
We prove the identity for the upper $\theta$-intermediate pressure. 
The lower case can be proved analogously.

It follows from the definition that
\[
\overline{P}_{\mu}(\boldsymbol{f},\varphi,\theta)
\le
\inf\bigl\{\overline{P}^{\mathrm{fin}}(\boldsymbol{f},Z,\varphi,\theta):\mu(Z)=1,~Z\subseteq X\bigr\}.
\]
To show the opposite inequality, let 
\(a=\overline{P}_{\mu}(\boldsymbol{f},\varphi,\theta)\).
For any \(s>0\), there exists \(\delta>0\) such that for any \(\varepsilon\in(0,\delta)\), we have
\[
\limsup_{N\to\infty}
M_{\mu}(\boldsymbol{f},a+s,\varphi,\varepsilon,N,\theta)=0.
\]
For each \(m\ge2\), set \(\varepsilon_m=\delta/m\).
Then there exists \(k_m\in\mathbb N\) such that
\[
M_{\mu}(\boldsymbol{f},a+s,\varphi,\varepsilon_m,N,\theta)\le \frac12,
\qquad \forall N\ge k_m.
\]
Hence, for every \(N\ge k_m\), there exists a finite \(\mu\)-cover
$\mathcal S_{N}^{(m)}=\{B_{n_i}(x_i,\varepsilon_m)\}_i$
with
\[
N\le n_i\le \frac{N}{\theta}\quad(\theta\in(0,1]),
\qquad n_i\ge N\quad(\theta=0),
\]
such that
\[
\sum_i 
\exp\Bigl(
-(a+s)n_i+\sup_{y\in B_{n_i}(x_i,\varepsilon_m)} S_{n_i}\varphi(y)
\Bigr)
<1.
\]
Define
\[
Z_\delta
=
\bigcap_{m\ge2}\;\bigcap_{N\ge k_m}
\bigcup_i B_{n_i}(x_i,\varepsilon_m).
\]
Then \(\mu(Z_\delta)=1\) and for each \(m\),
\[
\limsup_{N\to\infty}
M^{\mathrm{fin}}(Z_\delta,a+s,\varphi,\varepsilon_m,N,\theta)\le1.
\]
Hence, for each \(m\),
\[
\overline{P}^{\mathrm{fin}}(\boldsymbol{f},Z_\delta,\varphi,\varepsilon_m,\theta)
\le a+s.
\]
Letting \(m\to\infty\), we obtain
\[
\overline{P}^{\mathrm{fin}}(\boldsymbol{f},Z_\delta,\varphi,\theta)
\le a+s.
\]
Letting \(s\to0\) yields the result.
\end{proof}

\begin{theorem}
Let \( (X,d) \) be a compact metric space and \( \boldsymbol{f} \) be a sequence of continuous self-maps of \( X \).
For any compact subset $Z\subseteq X$, $\varphi \in C(X,\mathbb{R})$ and $\theta\in [0,1]$, we have 
\[
\begin{aligned}
\underline{P}(\boldsymbol{f},Z,\varphi,\theta) &= \sup\bigl\{\underline{P}_{\mu}(\boldsymbol{f},  \varphi, \theta) : \mu \in \mathcal{M}(X), ~\mu(Z) = 1\bigr\}, \\
\overline{P}(\boldsymbol{f},Z,\varphi,\theta) &= \sup\bigl\{\overline{P}_{\mu}(\boldsymbol{f},  \varphi, \theta) :  \mu \in \mathcal{M}(X), ~\mu(Z) = 1\bigr\}.
\end{aligned}
\]
\end{theorem}

\begin{proof}
We prove the statement for the upper $\theta$-intermediate pressure; the lower case is analogous.
The inequality
\[
\overline{P}(\boldsymbol{f},Z,\varphi,\theta)
\ge 
\sup\{\overline{P}_{\mu}(\boldsymbol{f},\varphi,\theta):\mu \in \mathcal{M}(X), ~\mu(Z)=1\}
\]
follows directly from the definition.

For the reverse inequality, we construct a measure on $X$ as follows.
Since $Z$ is compact, let $E=\{y_i\}_{i\in \mathcal{I}}$ be a countable dense subset of $Z$ and define a measure $\mu$ supported on $E$ by assigning positive weights $p_i>0$ to $y_i$ such that
$\sum_{i\in \mathcal{I}}p_i=1.$
For $\varepsilon>0$ and $N\in\mathbb N$, let
$\mathcal S^*=\{B_{n_i}(x_i,\varepsilon)\}_{i=1}^{k_N}$
be a finite $\mu$-cover with admissible lengths
\[
N\le n_i\le \frac{N}{\theta}\ (\theta\in(0,1]),\quad
n_i\ge N\ (\theta=0).
\]
Then $E\subset\bigcup_{B\in\mathcal S^*}B$.
Let $L=\max_i n_i$. By uniform continuity of the maps $f_1^j$ for $j=0,\dots,L-1$, there exists $0<\delta<\varepsilon$ such that
\[
d(x',x'')<\delta \Rightarrow d(f_1^j(x'),f_1^j(x''))<\varepsilon,\quad j=0,\dots,L-1.
\]
For each $x\in Z$, choose $y_x\in E$ with $d(x,y_x)<\delta$. Then there exists $i$ such that
\[
d(f_1^j(x),f_1^j(x_i))<2\varepsilon,\quad j=0,\dots,n_i-1,
\]
so that $x\in B_{n_i}(x_i,2\varepsilon)$.
Hence $\{B_{n_i}(x_i,2\varepsilon)\}_i$ is a finite cover of $Z$.
For any \(x\in X\) and \(n\in\mathbb N\),
\[
\sup_{y\in B_n(x,2\varepsilon)} S_n\varphi(y)
\le 
S_n\varphi(x)+n\omega(2\varepsilon)
\le 
\sup_{y\in B_n(x,\varepsilon)} S_n\varphi(y)+n\omega(2\varepsilon).
\]
It follows that
\[
M^{\mathrm{fin}}(Z,\alpha+\omega(2\varepsilon),\varphi,2\varepsilon,N,\theta)
\le
M_{\mu}(\boldsymbol{f},\alpha,\varphi,\varepsilon,N,\theta).
\]
Thus
\[
\overline{P}^{\mathrm{fin}}(\boldsymbol{f},Z,\varphi,2\varepsilon,\theta)-\omega(2\varepsilon)
\le
\overline{P}_{\mu}(\boldsymbol{f},\varphi,\varepsilon,\theta).
\]
Letting $\varepsilon\to0$ and using the equivalence 
\(\overline{P}^{\mathrm{fin}}=\overline{P}\) on compact subsets, we obtain
\[
\overline{P}(\boldsymbol{f},Z,\varphi,\theta)
\le
\overline{P}_{\mu}(\boldsymbol{f},\varphi,\theta).
\]
Thus the converse inequality follows.
\end{proof}

\section*{Acknowledgements}
The author is sincerely grateful to the anonymous referee for the valuable comments and suggestions that greatly improved the quality of this manuscript.

\section*{Declarations}

\section*{Funding}
This work was supported by the Science and Technology 
Research Program of Chongqing Municipal Education Commission (Grant No.KJQN202500802).

\section*{Competing Interests}

The author declares that there are no conflicts of interest regarding this paper.

\bibliography{references}

@article{ju2026entropy,
  author  = {Ju, Y.},
  title   = {Intermediate topological entropies for subsets of nonautonomous dynamical systems},
  journal = {J. Dyn. Control Syst.},
  volume  = {32},
  pages   = {16},
  year    = {2026}
}

@article{zhong2024variational,
  author  = {Zhong, Xing-Fu and Chen, Zhi-Jing},
  title   = {New variational principles of topological pressures},
  journal = {Qual. Theory Dyn. Syst.},
  volume  = {23},
  pages   = {231--254},
  year    = {2024}
}

@article{li2022notes,
  author  = {Li, C.-B. and Ye, Y.-L.},
  title   = {Some notes on the topological pressure of non-autonomous systems},
  journal = {Topol. Methods Nonlinear Anal.},
  volume  = {60},
  pages   = {305--326},
  year    = {2022}
}

@article{zhao2017formula,
  author  = {Zhao, Cao and Chen, Ercai and Hong, Xiucheng and Zhou, Xiaoyao},
  title   = {A formula of packing pressure of a factor map},
  journal = {Entropy},
  volume  = {19},
  number  = {10},
  pages   = {526},
  year    = {2017}
}

@article{oprocha2011dimensional,
  author  = {Oprocha, Piotr and Zhang, Guohua},
  title   = {Dimensional entropy over sets and fibres},
  journal = {Nonlinearity},
  volume  = {24},
  number  = {8},
  pages   = {2325--2346},
  year    = {2011}
}

@article{li2013corrigendum,
  author  = {Li, Qian and Chen, Ercai and Zhou, Xiaoyao},
  title   = {Corrigendum to: ``A note on topological pressure for non-compact sets of a factor map''},
  journal = {Chaos Solitons Fract.},
  volume  = {53},
  pages   = {75--77},
  year    = {2013}
}

@article{zhong2023variational,
  author  = {Zhong, Xing-Fu and Chen, Zhi-Jing},
  title   = {Variational principles for topological pressures on subsets},
  journal = {Nonlinearity},
  volume  = {36},
  number  = {2},
  pages   = {1168--1191},
  year    = {2023}
}

@article{yang2025topological,
  author  = {Yang, Zhongxuan and Huang, Xiaojun},
  title   = {Topological pressure of non-autonomous iterated function systems for non-compact sets},
  journal = {Qual. Theory Dyn. Syst.},
  volume  = {24},
  number  = {4},
  pages   = {178},
  year    = {2025}
}

@article{nazarian2024variational,
  author  = {Sarkooh, J.N},
  title   = {Variational principle for topological pressure on subsets of non-autonomous dynamical systems},
  journal = {Bull. Malays. Math. Sci. Soc.},
  volume  = {47},
  number  = {2},
  pages   = {64},
  year    = {2024}
}

@article{adler1965topological,
  title   = {Topological entropy},
  author  = {Adler, Roy L and Konheim, Alan G and McAndrew, M Harry},
  journal = {Trans. Amer. Math. Soc.},
  volume  = {114},
  pages   = {309--319},
  year    = {1965}
}

@article{bowen1971entropy,
  title   = {Entropy for group endomorphisms and homogeneous spaces},
  author  = {Bowen, R.},
  journal = {Trans. Amer. Math. Soc.},
  volume  = {153},
  pages   = {401--414},
  year    = {1971}
}

@article{bowen1973topological,
  title   = {Topological entropy for noncompact sets},
  author  = {Bowen, R.},
  journal = {Trans. Amer. Math. Soc.},
  volume  = {184},
  pages   = {125--136},
  year    = {1973}
}

@article{dinaburg1970correlation,
  title   = {A correlation between topological entropy and metric entropy},
  author  = {Dinaburg, E. I.},
  journal = {Dokl. Akad. Nauk SSSR},
  volume  = {190},
  pages   = {19--22},
  year    = {1970}
}

@article{li2024variational,
  author  = {Li, Chang-Bing},
  title   = {Variational principle for packing topological pressure of nonautonomous dynamical systems},
  journal={arXiv preprint arXiv:2412.17634},
  year    = {2024}
}

@article{chen2025nonautonomous,
  author  = {Chen, Zhuo and Miao, Jun Jie},
  title   = {Nonautonomous Dynamical Systems {I}: Topological Pressures and Entropies},
  journal = {arXiv preprint arXiv:2508.01363},
  year    = {2025}
}

@article{liu2026intermediate,
  author  = {Liu, Yu and Selmi, Bilel and Li, Zhiming},
  title   = {Intermediate topological entropies: Bridging entropy and dimension theories in dynamical systems},
  journal = {Chaos Solitons Fractals},
  volume  = {208},
  pages   = {118308},
  year    = {2026}
}

@article{chen2025nonautonomousII,
  author  = {Chen, Zhuo and Miao, Jun Jie},
  title   = {Nonautonomous Dynamical Systems {II}: Variational Principles},
  journal = {arXiv preprint arXiv:2502.21149},
  year    = {2025}
}

@article{falconer2020intermediate,
  title   = {Intermediate dimensions},
  author  = {Falconer, Kenneth J and Fraser, Jonathan M and Kempton, Tom},
  journal = {Math. Z.},
  volume  = {296},
  number  = {1},
  pages   = {813--830},
  year    = {2020}
}

@article{falconer2021intermediate,
  title={Intermediate dimensions: a survey},
  author={Falconer, Kenneth J},
  journal={Thermodynamic Formalism: CIRM Jean-Morlet Chair, Fall 2019},
  pages={469--493},
  year={2021},
  publisher={Springer}
}

@article{fang2012dimensions,
  title   = {Dimensions of stable sets and scrambled sets in positive finite entropy systems},
  author  = {Fang, Chun and Huang, Wen and Yi, Yingfei and Zhang, Pengfei},
  journal = {Ergodic Theory Dyn. Syst.},
  volume  = {32},
  number  = {2},
  pages   = {599--628},
  year    = {2012}
}

@article{feng2012variational,
  title   = {Variational principles for topological entropies of subsets},
  author  = {Feng, De-Jun and Huang, Wen},
  journal = {J. Funct. Anal.},
  volume  = {263},
  number  = {8},
  pages   = {2228--2254},
  year    = {2012}
}

@article{ruelle1973statistical,
  author  = {Ruelle, David},
  title   = {Statistical mechanics on a compact set with $\mathbb{Z}^v$ action satisfying expansiveness and specification},
  journal = {Trans. Amer. Math. Soc.},
  volume  = {185},
  pages   = {237--251},
  year    = {1973}
}

@article{tang2015variational,
  author  = {Tang, X. and Cheng, W.-C. and Zhao, Y.},
  title   = {Variational principle for topological pressures on subsets},
  journal = {J. Math. Anal. Appl.},
  volume  = {424},
  number  = {2},
  pages   = {1272--1285},
  year    = {2015}
}

@article{pesin1984topological,
  author  = {Pesin, Y. B. and Pitskel', B. S.},
  title   = {Topological pressure and the variational principle for noncompact sets},
  journal = {Funct. Anal. Appl.},
  volume  = {18},
  pages   = {307--318},
  year    = {1984}
}

@book{walters1982ergodic,
  author    = {Walters, Peter},
  title     = {An introduction to ergodic theory},
  publisher = {Springer-Verlag},
  address   = {New York},
  year      = {1982}
}

@article{kolyada1996topological,
  title   = {Topological entropy of nonautonomous dynamical systems},
  author  = {Kolyada, S. and Snoha, L.},
  journal = {Random Comput. Dyn.},
  volume  = {4},
  pages   = {205--233},
  year    = {1996}
}

@article{li2015remarks,
  title   = {Remarks on topological entropy of nonautonomous dynamical systems},
  author  = {Li, Zhiming},
  journal = {Int. J. Bifurcation Chaos},
  volume  = {25},
  number  = {12},
  year    = {2015}
}

@article{li2023comparison,
  title   = {A comparison of topological entropies for nonautonomous dynamical systems},
  author  = {Li, Chang-Bing and Ye, Yuan-Ling},
  journal = {J. Math. Anal. Appl.},
  volume  = {517},
  number  = {2},
  pages   = {126627},
  year    = {2023}
}

@book{pesin1997dimension,
  title     = {Dimension Theory in Dynamical Systems: Contemporary Views and Applications},
  author    = {Pesin, Y. B.},
  publisher = {Univ. Chicago Press},
  address   = {Chicago},
  year      = {1997}
}

@article{huang2008pressure,
  author  = {Huang, Xianjiu and Wen, Xi and Zeng, Fanping},
  title   = {Topological pressure of nonautonomous dynamical systems},
  journal = {Nonlinear Dyn. Syst. Theory},
  volume  = {8},
  pages   = {43--48},
  year    = {2008}
}

@article{kuang2013fractal,
  author  = {Kuang, Rui and Cheng, Wen-Chiao and Li, Bing},
  title   = {Fractal entropy of nonautonomous systems},
  journal = {Pacific J. Math.},
  volume  = {262},
  pages   = {421--436},
  year    = {2013}
}

@article{kong2015topological,
  author  = {Kong, Mengmeng and Cheng, Wen-Chiao and Li, Bing},
  title   = {Topological pressure for nonautonomous systems},
  journal = {Chaos Soliton. Fract.},
  volume  = {76},
  pages   = {82--92},
  year    = {2015}
}

@article{bis2018topological,
  author  = {Bi{\'s}, Andrzej},
  title   = {Topological and measure-theoretical entropies of nonautonomous dynamical systems},
  journal = {J. Dyn. Differ. Equ.},
  volume  = {30},
  number  = {1},
  pages   = {273--285},
  year    = {2018}
}

@article{xu2018variational,
  author  = {Xu, Leiye and Zhou, Xiaomin},
  title   = {Variational principles for entropies of nonautonomous dynamical systems},
  journal = {J. Dyn. Differ. Equ.},
  volume  = {30},
  number  = {3},
  pages   = {1053--1062},
  year    = {2018}
}

@article{li2019topological,
  author  = {Li, Zhiqiang and Zhang, Wenda and Wang, Wei},
  title   = {Topological entropy dimension for nonautonomous dynamical systems},
  journal = {J. Math. Anal. Appl.},
  volume  = {475},
  pages   = {1978--1991},
  year    = {2019}
}

@article{rodrigues2022mean,
  author  = {Rodrigues, Fagner B and Acevedo, Jeovanny Muentes},
  title   = {Mean dimension and metric mean dimension for nonautonomous dynamical systems},
  journal = {J. Dynam. Control Syst.},
  volume  = {28},
  pages   = {697--723},
  year    = {2022}
}

@article{zhang2023variational,
  author  = {Zhang, Ruifeng and Zhu, Jianghui},
  title   = {The variational principle for the packing entropy of nonautonomous dynamical systems},
  journal = {Acta Math. Sci.},
  volume  = {43},
  number  = {4},
  pages   = {1915--1924},
  year    = {2023}
}

@article{li2025dimension,
  author  = {Li, Chang-Bing},
  title   = {Topological entropy dimension on subsets for nonautonomous dynamical systems},
  journal = {J. Math. Anal. Appl.},
  volume  = {550},
  number  = {2},
  pages   = {129539},
  year    = {2025}
}

@article{shao2026topological,
  author  = {Shao, Hua},
  title   = {Topological sequence entropy of nonautonomous dynamical systems},
  journal = {J. Differential Equations},
  volume  = {453},
  pages   = {113923},
  year    = {2026}
}

@article{liu2022polynomial,
  author  = {Liu, Lei and Zhao, Cao},
  title   = {Polynomial entropy of nonautonomous dynamical systems for noncompact sets},
  journal = {J. Math. Anal. Appl.},
  volume  = {509},
  number  = {2},
  pages   = {125974},
  year    = {2022}
}

\end{document}